\documentclass{article}

\usepackage{amsmath}
\usepackage{amsthm}
\usepackage{amssymb}
\usepackage{amsfonts}

\usepackage{subeqnarray}
\usepackage{bm}
\usepackage{tikz}
\usepackage{dsfont}
\usepackage{upgreek}
\usepackage{booktabs}
\usepackage{fancyhdr}
\usepackage{multirow}
\usepackage{multicol}
\usepackage{float}
\usepackage[inline]{enumitem}   
\usepackage{cases}

\usepackage[ruled,vlined,linesnumbered,inoutnumbered,algo2e]{algorithm2e}

\usepackage{bbding}
\usepackage{caption}
\usepackage{todonotes}

\usepackage{mathrsfs}

\usepackage{lineno}

\usepackage[top=2.5cm, bottom=2.5cm, left=3cm, right=3cm]{geometry} 

\usepackage{graphics,graphicx,epsf,epstopdf,subfigure}
\graphicspath{{./Figures/}}
\ifpdf
\DeclareGraphicsExtensions{.eps,.pdf,.png,.jpg}
\else
\DeclareGraphicsExtensions{.eps}
\fi

\usepackage[breaklinks,colorlinks,linkcolor=black,citecolor=black,urlcolor=black]{hyperref}

\providecommand{\U}[1]{\protect\rule{.1in}{.1in}}

\newtheorem{theorem}{Theorem}[section]
\newtheorem{corollary}[theorem]{Corollary}
\newtheorem{lemma}[theorem]{Lemma}
\newtheorem{proposition}[theorem]{Proposition}
\newtheorem{definition}[theorem]{Definition}

\numberwithin{equation}{section}

\newcommand{\dist}{\mathrm{dist}}

\newcommand{\conv}{\mathrm{conv}}

\newcommand{\tr}{\mathrm{tr}}

\newcommand{\st}{\mathrm{s.\,t.}\,\,} 
\newcommand{\proj}{\mathrm{Proj}}

\newcommand{\retr}{\mathfrak{R}}
\newcommand{\grad}{\mathrm{grad}}

\newcommand{\bK}{\mathbb{K}}
\newcommand{\bN}{\mathbb{N}}
\newcommand{\bR}{\mathbb{R}}
\newcommand{\bS}{\mathbb{S}}
\newcommand{\bT}{\mathbb{T}}

\newcommand{\cM}{\mathcal{M}}

\newcommand{\cT}{\mathcal{T}}

\newcommand{\dE}{\mathds{E}}

\newcommand{\dN}{\mathds{N}}
\newcommand{\dP}{\mathds{P}}
\newcommand{\dQ}{\mathds{Q}}
\newcommand{\dW}{\mathds{W}}

\newcommand{\sB}{\mathscr{B}}
\newcommand{\sJ}{\mathscr{J}}
\newcommand{\sP}{\mathscr{P}}
\newcommand{\sQ}{\mathscr{Q}}
\newcommand{\sM}{\mathscr{M}}

\newcommand{\rmd}{\mathrm{d}}
\newcommand{\rmD}{\mathrm{D}}

\newcommand{\Rd}{\mathbb{R}^{d}}
\newcommand{\Rdd}{\mathbb{R}^{d \times d}}
\newcommand{\Rdr}{\mathbb{R}^{d \times r}}
\newcommand{\Rrr}{\mathbb{R}^{r \times r}}

\newcommand{\Odr}{\mathcal{O}^{d,r}}

\newcommand{\bSpd}{\mathbb{S}_{+}^{d}}
\newcommand{\bSppd}{\mathbb{S}_{++}^{d}}

\newcommand{\zz}{^{\top}}
\newcommand{\inv}{^{-1}}
\newcommand{\ff}{_{\mathrm{F}}}
\newcommand{\fs}{^2_{\mathrm{F}}}

\newcommand{\last}{_{\ast}}

\newcommand{\lcirc}{_{\circ}}
\newcommand{\udiamond}{^{\diamond}}

\newcommand{\uprime}{^{\prime}}

\newcommand{\lrmR}{_{\mathrm{R}}}

\newcommand{\dkh}[1]{\left(#1\right)}
\newcommand{\hkh}[1]{\left\{#1\right\}}
\newcommand{\fkh}[1]{\left[#1\right]}
\newcommand{\jkh}[1]{\left\langle#1\right\rangle}
\newcommand{\norm}[1]{\left\|#1\right\|}
\newcommand{\abs}[1]{\left\lvert #1\right\rvert}

\allowdisplaybreaks
\graphicspath{ {./results/} }
\definecolor{Gray}{rgb}{0.5,0.5,0.5}
\DeclareMathOperator*{\argmin}{arg\,min}

\makeatletter

\newcommand{\Rmnum}[1]{\expandafter\@slowromancap\romannumeral #1@}
\makeatother

\title{Enhancing Distributional Robustness in Principal \\ Component Analysis by Wasserstein Distances\thanks{This work is supported by RGC grant JLFS/P-501/24 for the CAS AMSS-PolyU Joint Laboratory in Applied Mathematics, Hong Kong Research Grant Council project PolyU15300024, National Natural Science Foundation of China (12125108, 11991021, 11991020, 12021001), and Key Research Program of Frontier Sciences, Chinese Academy of Sciences (ZDBS-LY-7022).}}

\author{
	Lei Wang\footnote{Department of Applied Mathematics, The Hong Kong Polytechnic University, Hong Kong, China (\href{mailto:wlkings@lsec.cc.ac.cn}{wlkings@lsec.cc.ac.cn}).}
	\and
	Xin Liu\footnote{State Key Laboratory of Scientific and Engineering Computing, Academy of Mathematics and Systems Science, Chinese Academy of Sciences, and University of Chinese Academy of Sciences, Beijing, China (\href{mailto:liuxin@lsec.cc.ac.cn}{liuxin@lsec.cc.ac.cn}).}
	\and
	Xiaojun Chen\footnote{Department of Applied Mathematics, The Hong Kong Polytechnic University, Hong Kong, China (\href{mailto:maxjchen@polyu.edu.hk}{maxjchen@polyu.edu.hk}).}
}

\date{}

\begin{document}

\maketitle

\begin{abstract}
	We consider the distributionally robust optimization (DRO) model of principal component analysis (PCA) to account for uncertainty in the underlying probability distribution. 
	The resulting formulation leads to a nonsmooth constrained min-max optimization problem, where the ambiguity set captures the distributional uncertainty by the type-$2$ Wasserstein distance. 
	We prove that the inner maximization problem admits a closed-form optimal value. 
	This explicit characterization equivalently reformulates the original DRO model into a minimization problem on the Stiefel manifold with intricate nonsmooth terms, a challenging formulation beyond the reach of existing algorithms. 
	To address this issue, we devise an efficient smoothing manifold proximal gradient algorithm. 
	Our analysis establishes Riemannian gradient consistency and global convergence of our algorithm to a stationary point of the nonsmooth minimization problem. 
	We also provide the iteration complexity $O(\epsilon^{-3})$ of our algorithm to achieve an $\epsilon$-approximate stationary point. 
	Finally, numerical experiments are conducted to validate the effectiveness and scalability of our algorithm, as well as to highlight the necessity and rationality of adopting the DRO model for PCA. 
\end{abstract}

%\noindent\rule{\textwidth}{0.05em}

% ---------------------------------------------------------------------------------------------------------------------------------

\section{Introduction}

Let $\xi \in \Rd$ be a $d$-dimensional random vector governed by a probability distribution $\dP\last$. 
In this paper, we consider the following distributionally robust optimization (DRO) model of principal component analysis (PCA),
\begin{equation} \label{opt:drspca} 
	\min_{X \in \Odr} \sup_{\dP \in \sP} \hspace{2mm}
	\dE_{\dP} \fkh{\norm{ \dkh{I_d - X X\zz} \dkh{\xi - \dE_{\dP} [\xi]} }\fs} + s (X).
\end{equation}
Here, the feasible set $\Odr := \{X \in \Rdr \mid X\zz X = I_r\}$, commonly referred to as the Stiefel manifold \cite{Absil2008optimization,Boumal2023introduction,Wang2021multipliers}, consists of all the $d \times r$ column-orthonormal matrices with $1 \leq r < d$. 
The ambiguity set $\sP$ represents a collection of distributions that could plausibly contain $\dP\last$ with high confidence. 
And $s: \Rdr \to \bR$ is a convex and Lipschitz continuous function acting as a regularizer to promote certain desired structures of solutions in $\Odr$, such as sparsity \cite{Chen2020proximal,Xiao2021exact} or nonnegativity \cite{Chen2025tight}.

In most practical scenarios, the underlying distribution $\dP\last$ is unknown and can not be captured precisely, leaving us without the essential information required to solve the PCA problem exactly. 
Although the sample-average approximation technique provides a practical model, its solutions often suffer from poor out-of-sample performances when the sample size is limited \cite{Esfahani2018data}. 
This dilemma motivates us to investigate the DRO model \eqref{opt:drspca} of PCA, which minimizes the worst case of the objective function across all distributions in $\sP$.

In the realm of DRO \cite{Goh2010distributionally,Kuhn2024distributionally}, there is a variety of ambiguity sets available, including those based on moment constraints \cite{Chen2019discrete,Delage2010distributionally,Xu2018distributionally}, divergences \cite{Van2021data,Wiesemann2013robust}, and Wasserstein distances \cite{Esfahani2018data,Gao2023distributionally}, among others. 
For a thorough and insightful exposition of these concepts, we refer interested readers to a recent survey \cite{Kuhn2024distributionally}, which offers an in-depth exploration of DRO problems.
Recently, Wasserstein DRO, where the discrepancy between probability measures is dictated by the Wasserstein distance, has garnered tremendous attentions across various domains \cite{Kuhn2019wasserstein,Kuhn2024distributionally}. 
In a similar vein, we explore the use of the Wasserstein distance in constructing the ambiguity set $\sP$ in problem \eqref{opt:drspca}. 
Let $\sQ_p$ be the space of all probability distributions $\dP$ supported on $\Rd$ with finite $p$-th moments, namely, $\dE_{\dP} (\norm{\xi}^p) = \int_{\Rd} \norm{\xi}^p \dP (\rmd \xi) < \infty$. 
Below is the definition of the Wasserstein distance defined on $\sQ_p$.

\begin{definition}[\cite{Kantorovich1958space}] \label{def:was}
	The type-$p$ Wasserstein distance $\dW_p: \sQ_p \times \sQ_p \to \bR_+$ between two probability distributions $\dP_1 \in \sQ_p$ and $\dP_2 \in \sQ_p$ is defined as
	\begin{equation*}
		\dW_p (\dP_1, \dP_2) 
		:= \inf_{\dQ \in \sJ (\dP_1, \dP_2)} \dkh{ \dE_{(\xi_1, \xi_2) \sim \dQ} \fkh{ \norm{\xi_1 - \xi_2}_p^p } }^{1/p},
	\end{equation*}
	where $\norm{\;\cdot\;}_p$ represents the $\ell_p$ norm on $\Rd$, and $\sJ (\dP_1, \dP_2)$ is the set containing all the joint distributions of $\xi_1$ and $\xi_2$ with marginals $\dP_1$ and $\dP_2$, respectively.
\end{definition}

Throughout this paper, we are primarily interested in the case where $p \in [1, 2]$, which is of significant importance both theoretically and practically \cite{Gao2023finite}. 
The corresponding Wasserstein DRO model of PCA can be formulated as
\begin{equation}
	\label{opt:drspca-was}
	\tag{\text{$\mathrm{P}_\mathrm{mM}$}}
	\begin{aligned}
		\min_{X \in \Odr} \sup_{\dP \in \sQ_p} 
		\hspace{2mm} & \dE_{\dP} \fkh{\norm{ \dkh{I_d - X X\zz} \dkh{\xi - \dE_{\dP} [\xi]} }\fs} + s (X) \\
		\st \hspace{2mm} & \dP \in \sB_{p} (\dP\lcirc, \rho) := \hkh{\dP \in \sQ_p \mid \dW_p (\dP, \dP\lcirc) \leq \rho},
	\end{aligned}
\end{equation}
where $\dP\lcirc \in \sQ_p$ is a nominal distribution conceived as being close to $\dP\last$ and $\rho \geq 0$ is a radius of the Wasserstein ball. 
It is noteworthy that the performance guarantees of the ambiguity set $\sB_{p} (\dP\lcirc, \rho)$ defined in \eqref{opt:drspca-was} can be inherited from the theoretical results in Wasserstein DRO. 
Notably, Esfahani and Kuhn \cite{Esfahani2018data} have provided an a priori estimate of the probability that the unknown true distribution $\dP\last$ resides in the Wasserstein ball. 
Blanchet et al. \cite{Blanchet2021sample,Blanchet2019robust} have developed a data-driven approach to construct a confidence region for the optimal choice of $\rho$. 
More recently, Gao \cite{Gao2023finite} has presented the finite-sample guarantees for generic Wasserstein DRO problems with $\rho = O (n^{-1/2})$ with $n \in \bN$ being the sample size, breaking the curse of dimensionality. 
These findings offer practical guidance in selecting an appropriate value of $\rho$.

By exploiting the structure of the inner maximization problem, we find that the DRO model \eqref{opt:drspca-was} is well-posed only when $p = 2$ as its optimal value becomes positive infinity for $1 \leq p < 2$.  
Moreover, the optimal value of the inner maximization problem can be computed in an explicit form for the case $p = 2$. 
This leads us to an equivalent reformulation of the DRO model \eqref{opt:drspca-was} with $p = 2$ as follows,
\begin{equation}
	\label{opt:was-eq-r}
	\tag{\text{$\mathrm{P}_\mathrm{m}$}}
	\min_{X \in \Odr} \hspace{2mm} 
	\tr \dkh{ \dkh{I_d - X X\zz} \Sigma\lcirc } + s (X) + 2 \rho \norm{ \dkh{I_d - X X\zz} \Sigma\lcirc^{1/2} }\ff
\end{equation}
where $\Sigma\lcirc = \dE_{\dP\lcirc} [ (\xi - \dE_{\dP\lcirc} [\xi]) (\xi - \dE_{\dP\lcirc} [\xi])\zz ] \in \Rdd$ is the covariance matrix of $\xi$ under the nominal distribution $\dP\lcirc$. 
It is evident that problem \eqref{opt:was-eq-r} is significantly easier to solve than problem \eqref{opt:drspca-was} in its original min-max form. 
Therefore, we turn our attention to developing efficient algorithms for tackling problem \eqref{opt:was-eq-r}.

There have been substantial advancements in nonsmooth optimization on Riemannian manifolds in recent decades, with the majority of efforts directed toward locally Lipschitz continuous functions. 
During this period, a diverse range of algorithms has emerged, including subgradient-oriented approaches \cite{Hosseini2018line,Hosseini2017riemannian,Li2021weakly}, proximal point methods \cite{Bento2017iteration,Chen2021manifold}, primal-dual frameworks \cite{Lai2014splitting,Zhang2020primal}, proximal gradient algorithms \cite{Chen2020proximal,Chen2024nonsmooth,Huang2022riemannian,Wang2022manifold}, proximal Newton methods \cite{Si2024riemannian}, infeasible approaches \cite{Hu2024constraint,Liu2024penalty,Xiao2021exact}, and so on. 
However, due to the presence of two nonsmooth terms, none of existing algorithms can be deployed to solve problem \eqref{opt:was-eq-r} efficiently. 
In particular, the last term in problem \eqref{opt:was-eq-r} poses a formidable challenge, as its proximal operator lacks a closed-form solution and entails costly computations. 
Since it can be represented by a composition of a smooth mapping and a convex function, we may resort to the proximal linear algorithm \cite{Wang2022manifold} for handling this issue. 
However, this approach requires to calculate the Jacobian of $(I_d - X X\zz) \Sigma\lcirc^{1 / 2}$, which involves the square root of $\Sigma\lcirc$. 
And the presence of another nonsmooth term renders the subproblem particularly difficult to tackle. 
The Riemannian subgradient method proposed in \cite{Li2021weakly} is capable of directly solving problem \eqref{opt:was-eq-r}. 
Nevertheless, owing to its reliance solely on subgradient information, it suffers from slow convergence with an iteration complexity of $O (\epsilon^{-4})$. 
Consequently, solving problem \eqref{opt:was-eq-r} remains a challenging endeavor.

Within the scope of this work, proximal gradient methods share the closest theoretical and methodological connection with the present study. 
This class of algorithms is initially proposed in \cite{Chen2020proximal} on the Stiefel manifold, which lays the foundation for a plethora of subsequent advancements \cite{Chen2024nonsmooth}. 
Later on, Huang and Wei \cite{Huang2022riemannian} extend this framework to general Riemannian manifolds. 
The crux of these approaches involves solving a proximal gradient subproblem on the tangent space. 
While it lacks a closed-form solution, this subproblem can be efficiently tackled by various numerical techniques, such as semi-smooth Newton methods \cite{Chen2020proximal} and fixed-point methods \cite{Liu2024penalty}.

The main contributions of this paper are summarized as follows. 

\noindent $\bullet$ 
Our investigation reveals that, the optimal value of the inner maximization problem in \eqref{opt:drspca-was} diverges to infinity when $1 \leq p < 2$, whereas for $p = 2$, it admits a closed-form expression. 
Accordingly, we shift our attention to the case $p = 2$ in the subsequent analysis. 
This explicit characterization facilitates an equivalent reformulation of the original DRO model \eqref{opt:drspca-was} into a minimization form \eqref{opt:was-eq-r}. 
A particularly intriguing discovery is that, for the classical PCA problem without regularizers, its distributionally robust counterpart is equivalent to the nominal model, uncovering an unexpected equivalence in this specific context. 

\noindent $\bullet$
We design a novel smoothing method to solve a class of nonsmooth optimization problems on the Stiefel manifold. 
Our theoretical analysis elucidates the Riemannian gradient consistency for the smoothing function and establishes the global convergence of our algorithm to a stationary point. 
Furthermore, we provide the iteration complexity of $O (\epsilon^{-3})$ required to achieve an $\epsilon$-approximate stationary point. 

\noindent $\bullet$
Last but not least, preliminary experimental results present the numerical performance of the proposed algorithm, underscoring its practical viability. 
Moreover, these results validate the necessity and rationality of adopting the DRO model for PCA, demonstrating its superiority in addressing the distributional uncertainty and enhancing the solution robustness through three real-world datasets.

The rest of this paper proceeds as follows. 
Section \ref{sec:preliminary} draws into some preliminaries of Riemannian optimization. 
In Section \ref{sec:model}, we make a profound study on the DRO model of PCA and derive its equivalent reformulation. 
Section \ref{sec:algorithm} discusses the stationarity condition and develops the smoothing algorithm. 
Its convergence analysis and iteration complexity are provided in Section \ref{sec:convergence}. 
Numerical results are presented in Section \ref{sec:experiment}. 
Finally, concluding remarks are given in Section \ref{sec:conclusion}.

\section{Preliminaries} 

\label{sec:preliminary}

In this section, we introduce the notations and concepts used throughout this paper.

\subsection{Basic Notations}

We use $\bR$ and $\bN$ to denote the sets of real and natural numbers, respectively.
And the notations $\bR_{+}$ and $\bR_{++}$ represent the sets of nonnegative and positive real numbers, respectively.
The Euclidean inner product of two matrices $Y_1, Y_2$ with the same size is defined as $\jkh{Y_1, Y_2}=\tr(Y_1\zz Y_2)$, where $\tr (B)$ stands for the trace of a square matrix $B$. 
And the notation $I_r \in \Rrr$ represents the $r \times r$ identity matrix.
The Frobenius norm of a given matrix $C$ is denoted by $\norm{C}\ff$. 
We denote by $\bSpd$ and $\bSppd$ the spaces of symmetric positive semidefinite matrices and symmetric positive definite matrices in $\Rdd$, respectively. 
The notation $B^{1/2}$ stands for the square root of a symmetric positive semidefinite matrix $B$.

\subsection{Riemannian Gradients and Clarke Subdifferentials}

Let $\cM$ be a complete submanifold embedded in $\Rdr$.
For each point $X \in \cM$, the tangent space to $\cM$ at $X$ is referred to as $\cT_{X} \cM$. 
In this paper, we consider the Riemannian metric $\jkh{\cdot, \cdot}_X$ on $\cT_{X} \cM$ that is induced from the Euclidean inner product $\jkh{\cdot, \cdot}$, i.e., $\jkh{V_1, V_2}_X = \jkh{V_1, V_2} = \tr(V_1\zz V_2)$ for any $V_1, V_2 \in \cT_{X} \cM$. 
The tangent bundle of $\cM$ is denoted by $\cT \cM = \{(X, V) \mid X \in \cM, V \in \cT_X \cM\}$, that is, the disjoint union of the tangent spaces of $\cM$. 
Additionally, we use the notation $\proj_{\cT_X \cM} (\cdot)$ to represent the orthogonal projection operator onto $\cT_X \cM$. 
In the context of the Stiefel manifold, the tangent space at $X \in \Odr$ can be described as $\cT_{X} \Odr = \{D \in \Rdr \mid X\zz D + D\zz X = 0\}$, and its orthogonal projection operator is given by $\proj_{\cT_X \Odr} (V) = V - X (X\zz V + V\zz X) / 2$ for any $V \in \Rdr$.

For a smooth function $f$, the Riemannian gradient at $X \in \cM$, denoted by $\grad\, f (X)$, is defined as the unique element of $\cT_X \cM$ satisfying 
\begin{equation*}
	\jkh{\grad\, f (X), V} = \rmD f (X) [V], \quad \forall V \in \cT_X \cM,
\end{equation*}
where $\rmD f (X) [V]$ is the directional derivative of $f$ along the direction $V$ at the point $X$.
Since $f$ is defined on an embedded submanifold in the Euclidean space, its Riemannian gradient can be computed by projecting the Euclidean gradient $\nabla f (X)$ onto the tangent space as follows,
\begin{equation*}
	\grad\, f (X) = \proj_{\cT_X \cM} \dkh{\nabla f (X)}.
\end{equation*}

For a locally Lipschitz continuous function on the manifold, the Riemannian Clarke subdifferential has been intensively studied and used in the literature \cite{Yang2014optimality,Hosseini2017riemannian,Hosseini2018line}, which is a natural extension of the Clarke subdifferential \cite{Clarke1990optimization,Rockafellar2009variational} in the Euclidean space. 
Throughout this paper, we adopt the following definition for the Riemannian Clarke subdifferential.

\begin{definition}[\cite{Hosseini2017riemannian}]
	Suppose that $f: \cM \to \bR$ is a locally Lipschitz continuous function.
	Let $\Omega\lrmR (f) = \hkh{ X \in \cM \mid \mbox{$f$ is differentiable at $X$}}$. 
	Then the Riemannian Clarke subdifferential of $f$ at $X \in \cM$ is defined as
	\begin{equation*}
		\partial\lrmR f (X) = \conv \hkh{D \in \cT_{X} \cM \mid \grad\, f (X_t) \to D, \Omega\lrmR (f) \ni X_t \to X}.
	\end{equation*}
\end{definition}

\subsection{Retractions}

In contrast to the Euclidean setting, the point $X + V$ does not lie in the manifold in general for $X \in \cM$ and $V \in \cT_{X} \cM$, due to the absence of a linear structure in $\cM$.
The interplay between $\cM$ and $\cT_{X} \cM$ is typically carried out via the exponential mappings, which are usually computationally intensive to evaluate in practice. 
As an alternative, the concept of retraction, a first-order approximation of the exponential mapping, is proposed in the literature \cite{Absil2008optimization,Boumal2023introduction} to alleviate the heavy computational burden.

\begin{definition}[\cite{Absil2008optimization}] \label{def:retr}
	A retraction on a manifold $\cM$ is a smooth mapping $\retr: \cT \cM \to \cM$, and for any $X \in \cM$, the restriction of $\retr$ to $\cT_X \cM$, denoted by $\retr_{X}$,
	satisfies the following two properties.
	\begin{enumerate}
		
		\item For any $X \in \cM$, it holds that $\retr_{X} (0_X) = X$, where $0_X$ is the zero vector in $\cT_{X} \cM$.
		
		\item The differential of $\retr_{X}$ at $0_X$, denoted by $\rmD \retr_{X} (0_X)$, is the identity map $\mathrm{id}_{\cT_{X} \cM}$ on $\cT_{X} \cM$.
		
	\end{enumerate}
\end{definition}

By leveraging the retraction $\retr_{X} (V)$, we can obtain a point by moving away from $X \in \cM$ along the direction $V \in \cT_{X} \cM$, while remaining on the manifold.
To this extent, it defines an update rule to preserve the feasibility. 
Following the proof of Lemma 2.7 in \cite{Boumal2018global}, we know that the retraction satisfies the following properties.

\begin{lemma}[\cite{Boumal2018global}] \label{le:retr}
	Let $\cM$ be a compact embedded submanifold of an Euclidean
	space. 
	There exist two constants $M_1 > 0$ and $M_2 > 0$ such that the following two relationships hold,
	\begin{equation*}
		\norm{\retr_{X} (V) - X}\ff \leq M_1 \norm{V}\ff,
	\end{equation*}
	and
	\begin{equation*}
		\norm{\retr_{X} (V) - \dkh{X + V}}\ff \leq M_2 \norm{V}\fs,
	\end{equation*}
	for any $X \in \cM$ and $V \in \cT_{X} \cM$.
\end{lemma}

There are various practical realizations of retractions on the Stiefel manifold, such as QR factorization, polar decomposition and Cayley transformation. 
We refer interested readers to \cite{Absil2008optimization,Boumal2023introduction,Gao2018new,Wen2013feasible} for more details.

\section{Model Analysis} 

\label{sec:model}

The DRO model \eqref{opt:drspca-was} of PCA is investigated in this section. 
We focus on the inner maximization problem in \eqref{opt:drspca-was} as follows,
\begin{equation} \label{opt:was-inner}
	\begin{aligned}
		\varphi (X) 
		:= {} & \sup_{\dP \in \sB_{p} (\dP\lcirc, \rho)} \dE_{\dP} \fkh{\norm{ \dkh{I_d - X X\zz} \dkh{\xi - \dE_{\dP} [\xi]} }\fs} \\
		= {} & \sup_{\dP \in \sB_{p} (\dP\lcirc, \rho) } \hkh{ \dE_{\dP} \fkh{ \tr \dkh{ \dkh{I_d - X X\zz}  \dkh{\xi \xi\zz - \dE_{\dP} \fkh{\xi} \dE_{\dP} \fkh{\xi}\zz} } } }.
	\end{aligned}
\end{equation}
The objective function of \eqref{opt:was-inner} is quadratic in the underlying distribution $\dP$ rather than linear, and hence, existing reformulations of Wasserstein DRO problems \cite{Chu2024wasserstein,Esfahani2018data,Gao2023distributionally,Zhang2024short} are not really applicable anymore. 
To navigate this challenge, we introduce an auxiliary variable $\mu \in \Rd$ and propose the following splitting formulation of \eqref{opt:was-inner},
\begin{equation} 
	\label{opt:was-inner-sp}
	\begin{aligned}
		\sup_{\mu \in \sM_{p} (\dP\lcirc, \rho)} \sup_{\dP \in \sQ_p} 
		\hspace{2mm} & \dE_{\dP} \fkh{ \tr \dkh{ \dkh{I_d - X X\zz}  \xi \xi\zz} } - \tr \dkh{ \dkh{I_d - X X\zz} \mu \mu\zz } \\
		\st \hspace{2mm} & \dW_p (\dP, \dP\lcirc) \leq \rho, 
		\quad
		\dE_{\dP} \fkh{\xi} = \mu,
	\end{aligned}
\end{equation}
where $\sM_{p} (\dP\lcirc, \rho) := \hkh{\mu = \dE_{\dP} \fkh{\xi} \mid \dW_p (\dP, \dP\lcirc) \leq \rho}$. 
The constraint $\mu \in \sM_{p} (\dP\lcirc, \rho)$ is imposed to avoid an empty feasible set and to ensure the well-posedness of problem \eqref{opt:was-inner-sp}. 
For fixed $\mu \in \sM_{p} (\dP\lcirc, \rho)$ and $X \in \Odr$, we define
\begin{equation}
	\label{eq:psi}
	\psi (\mu, X) := \sup_{\dP \in \sQ_p} \hkh{ \dE_{\dP} \fkh{ \omega_{X} (\xi) } \;\middle|\; \dW_p (\dP, \dP\lcirc) \leq \rho, \dE_{\dP} \fkh{\xi} = \mu },
\end{equation}
where $\omega_{X} (\xi) := \tr \dkh{ \dkh{I_d - X X\zz}  \xi \xi\zz}$. 
Then it holds that
\begin{equation} \label{eq:varphi-psi}
	\varphi (X) = \sup_{\mu \in \sM_{p} (\dP\lcirc, \rho)} \hkh{ \psi (\mu, X) - \tr \dkh{ \dkh{I_d - X X\zz} \mu \mu\zz } }.
\end{equation}
Based on the preceding constructions, the objective function of the outer minimization problem in the Wasserstein DRO model \eqref{opt:drspca-was} can be expressed as $\varphi (X) + s (X)$.

\subsection{Dual Representation}

In this subsection, we aim to derive the dual representation of $\psi (\mu, X)$ defined in \eqref{eq:psi}. 
This part of analysis follows the idea of Zhang et al. \cite{Zhang2024short} based on the Legendre transform \cite{Rockafellar1970convex}. 
We generalize existing results by handling an additional equality constraint $\dE_{\dP} \fkh{\xi} = \mu$. 
Although our focus is primarily on PCA, the techniques we propose can be naturally extended to more general settings.

The following lemma reveals some useful properties of the function $\bar{\psi}: \bR_{+} \to \bR \cup \{+ \infty\}$ defined as
\begin{equation*}
	%\label{eq:psi-bar}
	\bar{\psi} (\tau) := \sup_{\dP \in \sQ_p} \hkh{ \dE_{\dP} \fkh{ \omega_{X} (\xi) } \;\middle|\; \dW_p^p (\dP, \dP\lcirc) \leq \tau, \dE_{\dP} \fkh{\xi} = \mu },
\end{equation*}
for fixed $\mu \in \sM_{p} (\dP\lcirc, \rho)$ and $X \in \Odr$.

\begin{lemma}
	\label{le:psi-prop}
	The function $\bar{\psi}$ is bounded from below, monotonically increasing, and concave on $\bR_{+}$. 
\end{lemma}

\begin{proof}
	The proof of this lemma follows along the same lines as that of \cite[Lemma 1]{Zhang2024short} and is therefore omitted for brevity. 
\end{proof}

For a function $h: \bR \to \bR \cup \{+\infty\}$, we denote by $h\udiamond: \bR \to \bR \cup \{+\infty\}$ its Legendre transform $h\udiamond (\lambda) := \sup_{\tau \in \bR} \{\lambda \tau - h (\tau)\}$. 
Then the dual representation of $\psi (\mu, X)$ can be constructed by resorting to the Legendre transform.

\begin{theorem} \label{thm:psi}
	For any $p \in [1, 2]$ and $\rho > 0$, it holds that
	\begin{equation}
		\label{eq:psi-dual}
		\psi (\mu, X) = \inf_{\lambda \geq 0, \varsigma \in \Rd} \hkh{ \lambda \rho^p + \varsigma\zz \mu + \dE_{\xi\lcirc \sim \dP\lcirc} \fkh{ \sup_{\xi \in \Rd} \bar{\omega}_{X} (\xi, \xi\lcirc) } },
	\end{equation}
	where $\bar{\omega}_{X} (\xi, \xi\lcirc) := \omega_{X} (\xi) - \lambda \norm{\xi - \xi\lcirc}_p^p - \varsigma\zz \xi$. 
\end{theorem}

\begin{proof}
	Let $\lambda \in \bR$. 
	If $\lambda < 0$, we have $(- \bar{\psi})\udiamond (- \lambda) = \sup_{\tau \geq 0} \hkh{ - \lambda \tau + \bar{\psi} (\tau) } \geq \sup_{\tau \geq 0} \hkh{ - \lambda \tau + \bar{\psi} (0) } = + \infty$. 
	Then our focus is on the case where $\lambda \geq 0$. 
	Taking the Legendre transform of $- \bar{\psi}$ leads to that
	\begin{equation*}
		\begin{aligned}
			(- \bar{\psi})\udiamond (- \lambda)
			= {} & \sup_{\tau \geq 0} \hkh{ - \lambda \tau + \bar{\psi} (\tau) } \\
			= {} & \sup_{\tau \geq 0} \sup_{\dP \in \sQ_p} \hkh{ \dE_{\dP} \fkh{ \omega_{X} (\xi) } - \lambda \tau \;\middle|\; \dW_p^p (\dP, \dP\lcirc) \leq \tau, \dE_{\dP} \fkh{\xi} = \mu } \\
			= {} & \sup_{\dP \in \sQ_p} \sup_{\tau \geq 0} \hkh{ \dE_{\dP} \fkh{ \omega_{X} (\xi) } - \lambda \tau \;\middle|\; \dW_p^p (\dP, \dP\lcirc) \leq \tau, \dE_{\dP} \fkh{\xi} = \mu } \\
			= {} & \sup_{\dP \in \sQ_p} \hkh{ \dE_{\dP} \fkh{ \omega_{X} (\xi) } - \lambda \dW_p^p (\dP, \dP\lcirc) \;\middle|\; \dE_{\dP} \fkh{\xi} = \mu }.
		\end{aligned}
	\end{equation*}
	According to Definition \ref{def:was}, it follows that
	\begin{equation*}
		\begin{aligned}
			(- \bar{\psi})\udiamond (- \lambda)
			%= {} & \sup_{\dP \in \sQ_p} \hkh{ \dE_{\dP} \fkh{ \omega_{X} (\xi) } - \lambda \dW_p^p (\dP, \dP\lcirc) \;\middle|\; \dE_{\dP} \fkh{\xi} = \mu } \\
			= {} & \sup_{\dP \in \sQ_p} \hkh{ \dE_{\dP} \fkh{ \omega_{X} (\xi) } - \lambda \inf_{\dQ \in \sJ (\dP, \dP\lcirc)} \dE_{(\xi, \xi\lcirc) \sim \dQ} \fkh{ \norm{\xi - \xi\lcirc}_p^p } \;\middle|\; \dE_{\dP} \fkh{\xi} = \mu } \\
			= {} & \sup_{\dP \in \sQ_p, \dQ \in \sJ (\dP, \dP\lcirc)} \hkh{ \dE_{\dP} \fkh{ \omega_{X} (\xi) } - \lambda \dE_{(\xi, \xi\lcirc) \sim \dQ} \fkh{ \norm{\xi - \xi\lcirc}_p^p } \;\middle|\; \dE_{\dP} \fkh{\xi} = \mu } \\
			= {} & \sup_{\dQ \in \bar{\sJ} (\dP\lcirc)} \hkh{ \dE_{(\xi, \xi\lcirc) \sim \dQ} \fkh{ \omega_{X} (\xi) - \lambda \norm{\xi - \xi\lcirc}_p^p } \;\middle|\; \dE_{(\xi, \xi\lcirc) \sim \dQ} \fkh{\xi} = \mu },
		\end{aligned}
	\end{equation*}
	where $\bar{\sJ} (\dP\lcirc)$ stands for the set containing all the joint distributions of $\xi$ and $\xi\lcirc$ with second marginal $\dP\lcirc$. 
	As a direct consequence of \cite[Proposition 2.1]{Xu2018distributionally}, the Slater condition holds for the above problem, and hence, the strong duality prevails. 
	Hence, we can proceed to show that
	\begin{equation*}
		\begin{aligned}
			(- \bar{\psi})\udiamond (- \lambda)
			= {} & \sup_{\dQ \in \bar{\sJ} (\dP\lcirc)} \hkh{ \dE_{(\xi, \xi\lcirc) \sim \dQ} \fkh{ \omega_{X} (\xi) - \lambda \norm{\xi - \xi\lcirc}_p^p } \;\middle|\; \dE_{(\xi, \xi\lcirc) \sim \dQ} \fkh{\xi} = \mu } \\
			= {} & \inf_{\varsigma \in \Rd} \hkh{ \varsigma\zz \mu + \sup_{\dQ \in \bar{\sJ} (\dP\lcirc)} \dE_{(\xi, \xi\lcirc) \sim \dQ} \fkh{ \omega_{X} (\xi) - \lambda \norm{\xi - \xi\lcirc}_p^p - \varsigma\zz \xi } } \\
			= {} & \inf_{\varsigma \in \Rd} \hkh{ \varsigma\zz \mu + \sup_{\dQ \in \bar{\sJ} (\dP\lcirc)} \dE_{(\xi, \xi\lcirc) \sim \dQ} \fkh{ \bar{\omega}_{X} (\xi, \xi\lcirc) } },
		\end{aligned}
	\end{equation*}
	where $\varsigma \in \Rd$ is the Lagrangian multiplier associated with the equality constraint $\dE_{(\xi, \xi\lcirc) \sim \dQ} \fkh{\xi} = \mu$. 
	Lemma \ref{le:psi-prop} illustrates that $\bar{\psi}$ is bounded from below, monotonically increasing, and concave in $\bR_{+}$. 
	Hence, either $\bar{\psi} (\tau) < + \infty$ for all $\tau \geq 0$ or $\bar{\psi} (\tau) = + \infty$ for all $\tau > 0$. 
	In the former case, by invoking the result of \cite[Theorem 12.2]{Rockafellar1970convex}, we can obtain that
	\begin{equation}
		\label{eq:conjugate}
		\begin{aligned}
			\bar{\psi} (\tau) 
			= {} & - (- \bar{\psi})^{\diamond\diamond} (\tau)
			= - \sup_{\lambda \in \bR} \hkh{ - \lambda \tau - (- \bar{\psi})\udiamond (- \lambda) }
			= \inf_{\lambda \geq 0} \hkh{ \lambda \tau + (- \bar{\psi})\udiamond (- \lambda) } \\
			= {} & \inf_{\lambda \geq 0, \varsigma \in \Rd} \hkh{ \lambda \tau + \varsigma\zz \mu + \sup_{\dQ \in \bar{\sJ} (\dP\lcirc)} \dE_{(\xi, \xi\lcirc) \sim \dQ} \fkh{ \bar{\omega}_{X} (\xi, \xi\lcirc) } },
		\end{aligned}
	\end{equation}
	for all $\tau > 0$. 
	In the latter case, it holds that $(- \bar{\psi})\udiamond (- \lambda) = + \infty$ for all $\lambda \geq 0$, which indicates that the relationship \eqref{eq:conjugate} is also valid. 
	According to \cite[Proposition 2]{Zhang2024short}, the function $\bar{\omega}_{X}$ satisfies the interchangability principle. 
	Then it follows that
	\begin{equation*}
		\bar{\psi} (\tau)
		= \inf_{\lambda \geq 0, \varsigma \in \Rd} \hkh{ \lambda \tau + \varsigma\zz \mu + \dE_{\xi\lcirc \sim \dP\lcirc} \fkh{ \sup_{\xi \in \Rd} \bar{\omega}_{X} (\xi, \xi\lcirc) } }.
	\end{equation*}
	The proof is completed by noting that $\psi (\mu, X) = \bar{\psi} (\rho^p)$. 
\end{proof}

Theorem \ref{thm:psi} also implies a remarkable result that the optimal value of the inner maximization problem in the DRO model \eqref{opt:drspca-was} is always infinity for any $p \in [1, 2)$ and $\rho > 0$. 

\begin{corollary}
	Suppose that $p \in [1, 2)$ and $\rho > 0$. 
	Then it holds that $\varphi (X) = + \infty$ for any $X \in \Odr$. 
\end{corollary}

\begin{proof}
	For any $p \in [1, 2)$, we have
	\begin{equation*}
		\sup_{\xi \in \Rd} \bar{\omega}_X (\xi, \xi\lcirc)
		= \sup_{\xi \in \Rd} \hkh{ \tr \dkh{ \dkh{I_d - X X\zz}  \xi \xi\zz} - \lambda \norm{\xi - \xi\lcirc}_p^p - \varsigma\zz \xi }
		= + \infty,
	\end{equation*}
	which together with Theorem \ref{thm:psi} infers that $\psi (\mu, X) = + \infty$. 
	From the relationship \eqref{eq:varphi-psi}, it can be deduced that $\varphi (X) = + \infty$. 
	We complete the proof. 
\end{proof}

\subsection{Equivalent Reformulation for $p = 2$}

This subsection is devoted to deriving the equivalent reformulation \eqref{opt:was-eq-r} of the DRO model \eqref{opt:drspca-was} for the specific situation where $p = 2$. 
To this end, we establish that the optimal value $\varphi (X)$ of problem \eqref{opt:was-inner} admits a closed-form formulation.

The following lemma first shows that the supremum $\sup\, \{ \bar{\omega}_{X} (\xi, \xi\lcirc) \mid \xi \in \Rd \}$ in the dual representation \eqref{eq:psi-dual} of $\psi (\mu, X)$ can be explicitly computed. 

\begin{lemma} \label{le:sup-xi}
	Suppose that $p = 2$ and $\rho > 0$. 
	If $\lambda > 1$, it holds that
	\begin{equation*}
		\dE_{\xi\lcirc \sim \dP\lcirc} \fkh{ \sup_{\xi \in \Rd} \bar{\omega}_{X} (\xi, \xi\lcirc) }
		= \dfrac{\lambda}{\lambda - 1} \tr \dkh{ \dkh{I_d - X X\zz} \dE_{\dP\lcirc} \fkh{\xi\lcirc \xi \lcirc\zz} } 
		+ \theta_X (\lambda, \varsigma) 
		- \varsigma\zz \mu,
	\end{equation*}
	where the function $\theta_X$ is defined as
	\begin{equation*}
		\theta_X (\lambda, \varsigma)
		= \dfrac{1}{4 \lambda (\lambda - 1)} \varsigma\zz \dkh{\lambda I_d - X X\zz} \varsigma 
		+ \varsigma\zz \dkh{\mu - \dfrac{1}{\lambda - 1} \dkh{\lambda I_d - X X\zz} \dE_{\dP\lcirc} \fkh{\xi\lcirc}}.
	\end{equation*}
	Moreover, if $\lambda \in [0, 1]$, we have
	\begin{equation*}
		\dE_{\xi\lcirc \sim \dP\lcirc} \fkh{ \sup_{\xi \in \Rd} \bar{\omega}_{X} (\xi, \xi\lcirc) } = + \infty.
	\end{equation*}
\end{lemma}

\begin{proof}
	Straightforward calculations yield that
	\begin{equation*}
		\begin{aligned}
			\bar{\omega}_{X} (\xi, \xi\lcirc)
			= {} & \omega_{X} (\xi) - \lambda \norm{\xi - \xi\lcirc}_2^2 - \varsigma\zz \xi \\
			= {} & - \xi\zz \dkh{ \dkh{\lambda - 1} I_d + X X\zz} \xi
			+ \dkh{2 \lambda \xi\lcirc - \varsigma}\zz \xi
			- \lambda \xi\lcirc\zz \xi\lcirc,
		\end{aligned}
	\end{equation*}
	which is a quadratic function with respect to $\xi \in \Rd$ for fixed $\xi\lcirc \in \Rd$. 
	We move on to investigate the following two cases. 
	
	\noindent {\bf Case I}: $\lambda \in [0, 1]$. 
	Since $r < d$, the matrix $(\lambda - 1) I_d + X X\zz$ has at least one nonpositive eigenvalue $\lambda - 1 \leq 0$ associated with the nonzero eigenvector $z \in \Rd$. 
	Then for any $t \in \bR$, we have
	\begin{equation*}
		\bar{\omega}_{X} (t z, \xi\lcirc)
		= t^2 (1 - \lambda)
		+ t \dkh{2 \lambda \xi\lcirc - \varsigma}\zz z
		- \lambda \xi\lcirc\zz \xi\lcirc.
	\end{equation*}
	Since $\lambda \in [0, 1]$, it holds that
	\begin{equation*}
		\sup_{t \in \bR}\, \bar{\omega}_{X} (t z, \xi\lcirc) = + \infty,
	\end{equation*}
	which further implies that
	\begin{equation*}
		\sup_{\xi \in \Rd} \bar{\omega}_{X} (\xi, \xi\lcirc) = + \infty.
	\end{equation*}
	
	\noindent {\bf Case II}: $\lambda > 1$. 
	In this case, the matrix $(\lambda - 1) I_d + X X\zz$ is positive definite. 
	Then $\bar{\omega}_{X} (\xi, \xi\lcirc)$ is strictly concave with respect to $\xi \in \Rd$ for fixed $\xi\lcirc \in \Rd$. 
	Hence, we can proceed to show that
	\begin{equation*}
		\sup_{\xi \in \Rd} \bar{\omega}_{X} (\xi, \xi\lcirc) 
		= \dkh{\lambda \xi\lcirc - \dfrac{1}{2} \varsigma}\zz \dkh{ \dkh{\lambda - 1} I_d + X X\zz}\inv \dkh{\lambda \xi\lcirc - \dfrac{1}{2} \varsigma}
		- \lambda \xi\lcirc\zz \xi\lcirc,
	\end{equation*}
	where the supremum is attained at $\xi = ((\lambda - 1) I_d + X X\zz)\inv (\lambda \xi\lcirc - \varsigma / 2)$. 
	Moreover, according to the Sherman–Morrison–Woodbury formula \cite[page 329]{Higham2008functions}, we have
	\begin{equation*}
		\dkh{(\lambda - 1) I_d + X X\zz}\inv 
		= \dfrac{1}{\lambda - 1} I_d - \dfrac{1}{\lambda (\lambda - 1)} X X\zz.
	\end{equation*}
	Then a straightforward verification reveals that
	\begin{equation*}
		\begin{aligned}
			\sup_{\xi \in \Rd} \bar{\omega}_{X} (\xi, \xi\lcirc) 
			= {} & \dfrac{\lambda}{\lambda - 1} \tr \dkh{ \dkh{I_d - X X\zz} \xi\lcirc \xi \lcirc\zz}
			- \dfrac{1}{\lambda - 1} \varsigma\zz \dkh{\lambda I_d - X X\zz} \xi\lcirc \\
			& + \dfrac{1}{4 \lambda (\lambda - 1)} \varsigma\zz \dkh{\lambda I_d - X X\zz} \varsigma,
		\end{aligned}
	\end{equation*}
	which completes the proof. 
\end{proof}

Leveraging the result of the previous lemma, we proceed to prove that $\psi (\mu, X)$ can be expressed in an explicit form. 
Recall that $\Sigma\lcirc = \dE_{\dP\lcirc} [ (\xi - \dE_{\dP\lcirc} [\xi]) (\xi - \dE_{\dP\lcirc} [\xi])\zz ]$ is the covariance matrix of $\xi$ under the nominal distribution $\dP\lcirc$.

\begin{lemma}
	\label{le:psi-expr}
	Suppose that $p = 2$ and $\rho > 0$. 
	Then, for any $\mu \in \sM_{2} (\dP\lcirc, \rho)$ and $X \in \Odr$, it holds that
	\begin{equation*}
		\begin{aligned}
			\psi (\mu, X) = {} & \dkh{ \dkh{ \tr \dkh{ \dkh{I_d - X X\zz} \Sigma\lcirc } }^{1/2} + \dkh{\rho^2 - \norm{\mu - \dE_{\dP\lcirc} \fkh{\xi\lcirc}}_2^2}^{1/2} }^2 \\
			& + \tr \dkh{ \dkh{I_d - X X\zz} \mu \mu\zz }.
		\end{aligned}
	\end{equation*}
\end{lemma}

\begin{proof}
	Based on Lemma \ref{le:sup-xi}, we can restrict our discussion to the case where $\lambda > 1$. 
	Moreover, it follows from Theorem \ref{thm:psi} that
	\begin{equation*}
		\psi (\mu, X)
		= \inf_{\lambda > 1} \hkh{
			\lambda \rho^2 + \dfrac{\lambda}{\lambda - 1} \tr \dkh{ \dkh{I_d - X X\zz} \dE_{\dP\lcirc} \fkh{\xi\lcirc \xi \lcirc\zz} } 
			+ \inf_{\varsigma \in \Rd} \theta_X (\lambda, \varsigma)
		},
	\end{equation*}
	It is clear that $\theta_X (\lambda, \varsigma)$ is a quadratic function with respect to $\varsigma \in \Rd$ for fixed $\lambda > 1$. 
	Since $\lambda I_d - X X\zz$ is positive definite, it holds that
	\begin{equation*}
		\begin{aligned}
			\inf_{\varsigma \in \Rd} \theta_X (\lambda, \varsigma)
			= {} & - \dfrac{\lambda}{\lambda - 1} \tr \dkh{ \dkh{I_d - X X\zz} \dE_{\dP\lcirc} \fkh{\xi\lcirc} \dE_{\dP\lcirc} \fkh{\xi\lcirc}\zz } \\
			& + \tr \dkh{ \dkh{I_d - X X\zz} \mu \mu\zz }
			- \lambda \norm{\mu - \dE_{\dP\lcirc} \fkh{\xi\lcirc}}_2^2,
		\end{aligned}
	\end{equation*}
	where the infimum is attained at $\varsigma = 2 \lambda \dE_{\dP\lcirc} \fkh{\xi\lcirc} - 2 \dkh{ \dkh{\lambda - 1} I_d + X X\zz } \mu$. 
	By simple calculations, we can obtain that
	\begin{equation*}
		\begin{aligned}
			& \lambda \rho^2 + \dfrac{\lambda}{\lambda - 1} \tr \dkh{ \dkh{I_d - X X\zz} \dE_{\dP\lcirc} \fkh{\xi\lcirc \xi \lcirc\zz} } 
			+ \inf_{\varsigma \in \Rd} \theta_X (\lambda, \varsigma) \\
			%		= {} & \lambda \dkh{\rho^2 - \norm{\mu - \dE_{\dP\lcirc} \fkh{\xi\lcirc}}_2^2}
			%		+ \dfrac{\lambda}{\lambda - 1} \tr \dkh{ \dkh{I_d - X X\zz} \dkh{\dE_{\dP\lcirc} \fkh{\xi\lcirc \xi \lcirc\zz} - \dE_{\dP\lcirc} \fkh{\xi\lcirc} \dE_{\dP\lcirc} \fkh{\xi\lcirc}\zz} } \\
			%		& + \tr \dkh{ \dkh{I_d - X X\zz} \mu \mu\zz } \\
			= {} & \dkh{\lambda - 1} \dkh{\rho^2 - \norm{\mu - \dE_{\dP\lcirc} \fkh{\xi\lcirc}}_2^2}
			+ \dfrac{1}{\lambda - 1} \tr \dkh{ \dkh{I_d - X X\zz} \Sigma\lcirc } \\
			& + \tr \dkh{ \dkh{I_d - X X\zz} \mu \mu\zz }
			+ \rho^2 - \norm{\mu - \dE_{\dP\lcirc} \fkh{\xi\lcirc}}_2^2
			+ \tr \dkh{ \dkh{I_d - X X\zz} \Sigma\lcirc }.
		\end{aligned}
	\end{equation*}
	For any $\mu \in \sM_{2} (\dP\lcirc, \rho)$, it follows from \cite[Theorom 2.1]{Gelbrich1990formula} that $\norm{\mu - \dE_{\dP\lcirc} \fkh{\xi\lcirc}}_2 \leq \dW_2 (\dP, \dP\lcirc) \leq \rho$. 
	Then it can be readily verified that
	\begin{equation*}
		\begin{aligned}
			\psi (\mu, X)
			= {} & \inf_{\lambda > 1} \hkh{\dkh{\lambda - 1} \dkh{\rho^2 - \norm{\mu - \dE_{\dP\lcirc} \fkh{\xi\lcirc}}_2^2} + \dfrac{1}{\lambda - 1} \tr \dkh{ \dkh{I_d - X X\zz} \Sigma\lcirc } } \\
			& + \tr \dkh{ \dkh{I_d - X X\zz} \mu \mu\zz }
			+ \rho^2 - \norm{\mu - \dE_{\dP\lcirc} \fkh{\xi\lcirc}}_2^2
			+ \tr \dkh{ \dkh{I_d - X X\zz} \Sigma\lcirc } \\
			%			= {} & 2 \dkh{ \dkh{\rho^2 - \norm{\mu - \dE_{\dP\lcirc} \fkh{\xi\lcirc}}_2^2} \tr \dkh{ \dkh{I_d - X X\zz} \Sigma\lcirc } }^{1/2} \\
			%			& + \rho^2 - \norm{\mu - \dE_{\dP\lcirc} \fkh{\xi\lcirc}}_2^2
			%			+ \tr \dkh{ \dkh{I_d - X X\zz} \Sigma\lcirc }
			%			+ \tr \dkh{ \dkh{I_d - X X\zz} \mu \mu\zz } \\
			= {} & \dkh{ \dkh{ \tr \dkh{ \dkh{I_d - X X\zz} \Sigma\lcirc } }^{1/2} + \dkh{\rho^2 - \norm{\mu - \dE_{\dP\lcirc} \fkh{\xi\lcirc}}_2^2}^{1/2} }^2 \\
			& + \tr \dkh{ \dkh{I_d - X X\zz} \mu \mu\zz }.
		\end{aligned}
	\end{equation*}
	We complete the proof. 
\end{proof}

We are now in a position to derive an explicit expression for $\varphi (X)$ based on the relationship \eqref{eq:varphi-psi}, as established in the following theorem.

\begin{theorem} \label{thm:varphi}
	For any $X \in \Odr$, $\dP\lcirc \in \sQ_2$, and $\rho \geq 0$, the optimal value of problem \eqref{opt:was-inner} with $p = 2$ has the following explicit expression,
	\begin{equation}
		\label{eq:varphi}
		\varphi (X) = \dkh{ \dkh{ \tr \dkh{ \dkh{I_d - X X\zz} \Sigma\lcirc } }^{1/2} + \rho }^2.
	\end{equation}
\end{theorem}

\begin{proof}
	We first consider the case where $\rho = 0$. 
	Then the feasible region $\sB_{2} (\dP\lcirc, 0)$ of problem \eqref{opt:was-inner} collapses to the singleton $\{\dP\lcirc\}$. 
	As a result, we have
	\begin{equation*}
		\varphi (X) = \tr \dkh{ \dkh{I_d - X X\zz} \Sigma\lcirc },
	\end{equation*}
	which indicates that the relationship $\eqref{eq:varphi}$ holds for $\rho = 0$. 
	
	Next, our focus is on the case where $\rho > 0$. 
	As a direct consequence of Lemma \ref{le:psi-expr}, we can proceed to show that
	\begin{equation*}
		\begin{aligned}
			\varphi (X)
			= {} & \sup_{\mu \in \sM_{2} (\dP\lcirc, \rho)}
			\hkh{ \psi (\mu, X) - \tr \dkh{ \dkh{I_d - X X\zz} \mu \mu\zz } } \\
			= {} & \sup_{\mu \in \sM_{2} (\dP\lcirc, \rho)}
			\dkh{ \dkh{ \tr \dkh{ \dkh{I_d - X X\zz} \Sigma\lcirc } }^{1/2} + \dkh{\rho^2 - \norm{\mu - \dE_{\dP\lcirc} \fkh{\xi\lcirc}}_2^2}^{1/2} }^2 \\
			= {} & \dkh{ \dkh{ \tr \dkh{ \dkh{I_d - X X\zz} \Sigma\lcirc } }^{1/2} + \rho }^2,
		\end{aligned}
	\end{equation*}
	where the supremum is attained at $\mu = \dE_{\dP\lcirc} \fkh{\xi\lcirc}$. 
	The proof is completed. 
\end{proof}

By expanding the square term in \eqref{eq:varphi}, it can be obtained that
\begin{equation*}
	\varphi (X)
	= \tr \dkh{ \dkh{I_d - X X\zz} \Sigma\lcirc }
	+ 2 \rho \dkh{ \tr \dkh{ \dkh{I_d - X X\zz} \Sigma\lcirc } }^{1/2}
	+ \rho^2.
\end{equation*}
It is crucial to recognize that the function $X \mapsto \dkh{ \tr \dkh{ \dkh{I_d - X X\zz} \Sigma\lcirc } }^{1/2}$ fails to be locally Lipschitz continuous in $\Rdr$ due to the presence of square roots. 
Fortunately, for any $X \in \Odr$, we have
\begin{equation*} %\label{eq:frobenius}
	\begin{aligned}
		\dkh{\tr \dkh{ \dkh{I_d - X X\zz} \Sigma\lcirc } }^{1/2}
		= {} & \dkh{\tr \dkh{ \Sigma\lcirc^{1/2} \dkh{I_d - X X\zz} \dkh{I_d - X X\zz} \Sigma\lcirc^{1/2} } }^{1/2} \\
		= {} & \norm{\dkh{I_d - X X\zz} \Sigma\lcirc^{1/2}}\ff.
	\end{aligned}
\end{equation*}
Based on this representation and Theorem \ref{thm:varphi}, the DRO model \eqref{opt:drspca-was} with $p = 2$ can be equivalently reformulated as problem \eqref{opt:was-eq-r}.

We end this section by demonstrating that, the solutions of classical PCA without regularizers inherently possess robustness against data perturbations, as characterized by the type-$2$ Wasserstein distance. 

\begin{corollary} %\label{coro:pca}
	Suppose that $p = 2$ and $s (X) = 0$ for all $X \in \Odr$. 
	Then for any $\Sigma\lcirc \in \bS_{+}^d$ and $\rho \geq 0$, the global minimizers of problem \eqref{opt:drspca-was} coincide with those of the following nominal model of PCA,
	\begin{equation*}
		\min_{X \in \Odr} \hspace{2mm} \tr \dkh{ \dkh{I_d - X X\zz} \Sigma\lcirc }.
	\end{equation*}
\end{corollary}

\begin{proof}
	According to Theorem \ref{thm:varphi}, we know that the DRO model \eqref{opt:drspca-was} is equivalent to problem \eqref{opt:was-eq-r}. 
	Then the nonnegativity of $\tr ( (I_d - X X\zz) \Sigma\lcirc )$ and $\rho$ results in that
	\begin{equation*}
		\argmin_{X \in \Odr} \hkh{ \dkh{ \dkh{ \tr \dkh{ \dkh{I_d - X X\zz} \Sigma\lcirc } }^{1/2} + \rho }^2 }
		%= {} & \dkh{ \min_{X \in \Odr} \hkh{ \dkh{ \tr \dkh{ \dkh{I_d - X X\zz} \Sigma\lcirc } }^{1/2} } + \rho }^2 \\
		= \argmin_{X \in \Odr} \hkh{ \tr \dkh{ \dkh{I_d - X X\zz} \Sigma\lcirc } },
	\end{equation*}
	%	\begin{equation*}
		%	\begin{aligned}
			%		\min_{X \in \Odr} \hkh{ \dkh{ \dkh{ \tr \dkh{ \dkh{I_d - X X\zz} \Sigma\lcirc } }^{1/2} + \rho }^2 }
			%		= {} & \dkh{ \min_{X \in \Odr} \hkh{ \dkh{ \tr \dkh{ \dkh{I_d - X X\zz} \Sigma\lcirc } }^{1/2} } + \rho }^2 \\
			%		= \dkh{ \dkh{ \min_{X \in \Odr} \hkh{ \tr \dkh{ \dkh{I_d - X X\zz} \Sigma\lcirc } } }^{1/2} + \rho }^2,
			%	\end{aligned}
		%	\end{equation*}
	which completes the proof. 
\end{proof}

\section{Algorithm Design} 

\label{sec:algorithm}

The purpose of this section is to devise an efficient algorithm to solve the equivalent reformulation \eqref{opt:was-eq-r} of the DRO model \eqref{opt:drspca-was} with $p = 2$. 
We consider a broader class of  nonsmooth optimization problems of the following form,
\begin{equation} \label{opt:main} 
	\min_{X \in \Odr} \hspace{2mm} 
	f (X) := u(X) + s (X) + w (X),
\end{equation}
where $u$, $s$, and $w$ are appropriate functions satisfying the following conditions. 
\begin{enumerate}
	
	\item The function $u: \Rdr \to \bR$ is continuously differentiable and its Euclidean gradient $\nabla u$ is Lipschitz continuous with the corresponding Lipschitz constant $L_u > 0$. 
	
	\item The function $s: \Rdr \to \bR$ is convex and Lipschitz continuous with the corresponding Lipschitz constant $L_s > 0$. 
	
	\item The function $w: \Rdr \to \bR$ is of the form $w (X) = 2 \rho \| (I_d - X X\zz) \Sigma\lcirc^{1/2} \|\ff$ with $\Sigma\lcirc \in \bS_{+}^d$ and $\rho > 0$. 
	
\end{enumerate}

It is evident that model \eqref{opt:was-eq-r} is a specific instance of problem \eqref{opt:main} by identifying $u (X) = \tr ( (I_d - X X\zz) \Sigma\lcirc )$. 
As a direct consequence of the continuity of $f$ over the compact manifold $\Odr$, there always exists an optimal solution of problem \eqref{opt:main}.

\subsection{Stationarity Condition}

In this subsection, we establish the stationarity condition for local minimizers of problem \eqref{opt:main}. 
According to the discussions in \cite{Hosseini2017riemannian,Yang2014optimality}, a necessary condition that $f$ achieves a local minimum at $X$ on $\Odr$ is that
\begin{equation*}
	0 \in \partial\lrmR f (X).
\end{equation*}
Since $u$ is smooth, $s$ is convex, and $w$ is a composition of a smooth mapping and a convex function, they are all weakly convex and hence regular \cite{Drusvyatskiy2019efficiency}. 
As a result, the objective function $f = u + s + w$ is also regular \cite{Clarke1990optimization}. 
Then it follows from \cite[Theorem 5.3]{Yang2014optimality} that
\begin{equation*}
	\partial\lrmR f (X) 
	= \grad\, u (X) + \partial\lrmR s (X) + \partial\lrmR w (X),
\end{equation*}
%\begin{equation*}
%	\begin{aligned}
	%		\partial\lrmR f (X) 
	%		= {} & \proj_{\cT_{X} \Odr} \dkh{\partial f (X)} \\
	%		= {} & \proj_{\cT_{X} \Odr} \dkh{\nabla u (X) + \partial s (X) + \partial w (X)} \\
	%		= {} & \grad\, u (X) + \partial\lrmR s (X) + \partial\lrmR w (X),
	%	\end{aligned}
%\end{equation*}
for any $X \in \Odr$.

Based on the above discussions, the stationarity condition of the nonsmooth problem \eqref{opt:main} can be stated as follows.

\begin{definition}
	A point $X\last \in \Odr$ is called a stationary point of problem \eqref{opt:main} if the following condition holds,
	\begin{equation*}
		0 \in \grad\, u (X\last) + \partial\lrmR s (X\last) + \partial\lrmR w (X\last).
	\end{equation*}
\end{definition}

\subsection{Smoothing Function}

To address the challenges posed by the nonsmooth term $w$, we propose to leverage the smoothing approximation technique \cite{Chen2012smoothing}. 
Specifically, we construct the smoothing function of $w$ as follows,
\begin{equation} \label{eq:fun-tw}
	\tilde{w} (X, \mu) = 
	\left\{
	\begin{aligned}
		& w (X), 
		&& \mbox{if~} w (X) \geq \mu \rho, \\
		& \dfrac{w^2 (X)}{2 \mu \rho} + \dfrac{\mu \rho}{2}, 
		&& \mbox{if~} w (X) < \mu \rho,
	\end{aligned}
	\right.
\end{equation}
where $\mu > 0$ is a smoothing parameter. 
Interested readers can refer to \cite{Chen2012smoothing} for more examples of smoothing functions. 

When it is clear from the context, the Euclidean and Riemannian gradients of $\tilde{w} (X, \mu)$ with respect to $X$ are simply denoted by $\nabla \tilde{w} (X, \mu)$ and $\grad\, \tilde{w} (X, \mu)$, respectively. The following proposition reveals that the smoothing function $\tilde{w}$ enjoys some favorable properties. 

\begin{proposition} \label{prop:smoothing}
	The function $\tilde{w}: \Rdr \times \bR_{++} \to \bR$ constructed in \eqref{eq:fun-tw} satisfies the  following conditions.
	\begin{enumerate}
		
		\item For any $\mu > 0$, $\tilde{w} (\cdot, \mu)$ is continuously differentiable over $\Rdr$.
		
		\item For any $X \in \Rdr$, it holds that
		\begin{equation*}
			\lim_{X\uprime \to X, \, \mu \downarrow 0} \tilde{w} (X\uprime, \mu) = w (X).
		\end{equation*}
		
		\item For any $X \in \Rdr$ and $\mu > 0$, we have
		\begin{equation*}
			w (X) 
			\leq \tilde{w} (X, \mu) 
			\leq w (X) + \dfrac{\mu \rho}{2}.
		\end{equation*}
		
		\item There exists a constant $M_w > 0$ such that $\norm{\nabla \tilde{w} (X, \mu)}\ff \leq M_w$ for any $X \in \Odr$ and $\mu > 0$.
		
		\item There exists a constant $L_w > 0$ such that, for any $\mu > 0$, $\nabla \tilde{w} (\cdot, \mu)$ is Lipschitz continuous over $\Odr$ with the corresponding Lipschitz constant $L_w \mu^{-1}$.
		
	\end{enumerate}
\end{proposition}

\begin{proof}
	The proof can be easily given, which is omitted here. 
\end{proof}

We adopt the smoothing function $\tilde{w}$ given in \eqref{eq:fun-tw} for two key reasons. 
First, the evaluation of both function values and gradients for $\tilde{w}$ circumvents the need to compute $\Sigma\lcirc^{1 / 2}$. 
Second, this particular smoothing function satisfies Riemannian gradient consistency, which will be rigorously demonstrated in the next subsection. 
This property is crucial for guaranteeing the global convergence of our algorithm.

\subsection{Riemannian Subdifferential}

The algorithm proposed in this paper is based on the smoothing approximation technique. 
Thus, it is natural that the convergence result is closely tied to the specific smoothing function employed. 
Below is the definition of the Riemannian subdifferential associated with the smoothing function, which serves as a fundamental concept in our analysis.

\begin{definition}
	The Riemannian subdifferential of $w$ associated with the smoothing function $\tilde{w}$ at $X \in \Odr$ is defined as
	\begin{equation*}
		\partial\lrmR |_{\tilde{w}} w (X) = \hkh{G \in \cT_X \Odr \mid \grad\, \tilde{w} (X_t, \mu_t) \to G, \Odr \ni X_t \to X, \mu_t \downarrow 0}. 
	\end{equation*}
\end{definition}

The following theorem establishes the Riemannian gradient consistency of the smoothing function $\tilde{w}$. 
%which constructs a crucial connection between the Riemannian subdifferential associated with a smoothing function and the Riemannian Clarke subdifferential. 
By bridging two subdifferentials, this property plays a pivotal role in showing the global convergence of the proposed algorithm to a stationary point of problem \eqref{opt:main}. 

\begin{theorem} \label{thm:consistency}
	For the smoothing function $\tilde{w}$ constructed in \eqref{eq:fun-tw}, it holds that
	\begin{equation*}
		\partial\lrmR |_{\tilde{w}} w (X) \subseteq \partial\lrmR w (X),
	\end{equation*}
	for any $X \in \Odr$.
\end{theorem}

\begin{proof}
	We fix an arbitrary $X \in \Odr$. 
	Let $G \in \partial\lrmR |_{\tilde{w}} w (X)$.
	Then there exist two sequences $\{X_t\} \subseteq \Odr$ and $\{\mu_t\} \subseteq \bR_{+}$ with $X_t \to X$ and $\mu_t \downarrow 0$ as $t \to \infty$ such that
	\begin{equation*}
		G = \lim_{\Odr \ni X_t \to X, \, \mu_t \downarrow 0} \grad\, \tilde{w} (X_t, \mu_t).
	\end{equation*}
	For convenience, we define the index set $\bT := \{t \in \bN \mid w (X_t) < \mu_t \rho\}$.
	
	If $\bT$ is a finite set, 
	there exists $\bar{t} \in \bN$ such that
	\begin{equation*}
		w (X_t) \geq \mu_t \rho > 0,
	\end{equation*}
	for any $t \geq \bar{t}$.
	Hence, the function $w$ is continuously differentiable near $X_t$ and we have
	\begin{equation*}
		\grad\, \tilde{w} (X_t, \mu_t)
		= \grad\, w (X_t),
	\end{equation*}
	for all $t \geq \bar{t}$. 
	Then it can be obtained that
	\begin{equation*}
		G = \lim_{\Odr \ni X_t \to X, \, \mu_t \downarrow 0} \grad\, \tilde{w} (X_t, \mu_t)
		= \lim_{\Odr \ni X_t \to X} \grad\, w (X_t),
	\end{equation*}
	which indicates that $G \in \partial\lrmR w (X)$.
	
	Next, we consider the case that $\bT$ is an infinite set.
	Then it is clear that $w (X) = 0$.
	Thus, the function $w$ attains the global minimum at $X$, which further implies that $0 \in \partial\lrmR w (X)$.
	Let $\bT\uprime := \{t \in \bN \mid w (X_t) = 0\}$ be a subset of $\bT$.
	If $\bT\uprime$ is also an infinite set, we can obtain that
	\begin{equation*}
		G = \lim_{\Odr \ni X_t \to X, \, \mu_t \downarrow 0, \, t \in \bT\uprime} \grad\, \tilde{w} (X_t, \mu_t).
	\end{equation*}
	Straightforward calculations yield that $\grad\, \tilde{w} (X_t, \mu_t) = 0 \in \partial\lrmR w (X_t)$ for any $t \in \bT\uprime$.
	Hence, the above relationship indicates that $G = 0 \in \partial\lrmR w (X)$.
	Now we assume that $\bT\uprime$ is a finite set.
	Then there exists $\hat{t} \in \bN$ such that
	\begin{equation*}
		0 < w (X_t) < \mu_t \rho,
	\end{equation*}
	for any $t \geq \hat{t}$.
	Moreover, the function $w$ is continuously differentiable near $X_t$ and we have
	\begin{equation} \label{eq:grad_h}
		\grad\, \tilde{w} (X_t, \mu_t) 
		= \tau_t \grad\, w (X_t),
	\end{equation}
	where $\tau_t$ is a constant defined by
	\begin{equation*} %\label{eq:tau_t}
		\tau_t = \dfrac{w (X_t)}{\mu_t \rho} \in (0, 1).
	\end{equation*}
	A straightforward verification reveals that
	\begin{equation*}
		\norm{\grad\, w (X_t)}\ff
		= 2 \rho \dfrac{\norm{(I_d - X_t X_t\zz) \Sigma\lcirc X_t}\ff}{\norm{(I_d - X_t X_t\zz) \Sigma\lcirc^{1/2}}\ff} \\
		\leq 2 \rho \norm{\Sigma\lcirc^{1/2}}\ff,
	\end{equation*}
	for any $X_t \in \Odr$ satisfying $w (X_t) \neq 0$. 
	Hence, the sequence $\{\grad\, w (X_t)\}_{t \geq \hat{t}}$ is bounded. 
	By passing to a subsequence if necessary, we may assume without loss of generality that
	\begin{equation*}
		H = \lim_{\Odr \ni X_t \to X} \grad\, w (X_t).
	\end{equation*}
	According to the definition of Riemannian Clarke subdifferentials, it holds that $H \in \partial\lrmR w (X)$. 
	Since $\tau_t \in (0, 1)$ for any $t \geq \hat{t}$, we can assume without loss of generality that $\lim_{t \to \infty} \tau_t = \tau$ for a constant $\tau \in [0, 1]$. 
	Consequently, it follows from the relationship \eqref{eq:grad_h} and the fact $0 \in \partial\lrmR w (X)$ that
	\begin{equation*}
		\begin{aligned}
			G = {} & \lim_{\Odr \ni X_t \to X, \, \mu_t \downarrow 0} \grad\, \tilde{w} (X_t, \mu_t)
			= \lim_{\Odr \ni X_t \to X, \, \mu_t \downarrow 0} \tau_t \grad\, w (X_t) \\
			= {} & \tau H
			= \tau H + (1 - \tau) 0
			\in \conv \hkh{\partial\lrmR w (X)}
			= \partial\lrmR w (X).
		\end{aligned}
	\end{equation*}
	The proof is completed.
\end{proof}

\subsection{Algorithm Development}

Based on the smoothing function $\tilde{w}$, we can obtain the following approximation of the objective function $f$ in problem \eqref{opt:main},
\begin{equation*}
	\tilde{f} (X, \mu) := \tilde{g} (X, \mu) + s (X),
\end{equation*}
where $\tilde{g}$ is given by
\begin{equation*}
	\tilde{g} (X, \mu) := u (X) + \tilde{w} (X, \mu).
\end{equation*}
It is clear that the function $\tilde{f} (X, \mu)$ exhibits a composite structure. 
In particular, the first term $\tilde{g} (X, \mu)$ is smooth for fixed $\mu > 0$, whereas the second term $s (X)$ is possibly nonsmooth. 
This inherent structure naturally lends itself to the framework of the proximal gradient method on the manifold to minimize $\tilde{f} (\cdot, \mu)$ over $\Odr$.

Specifically, we intend to solve the following subproblem to find the descent direction $V_k \in \cT_{X_k} \Odr$ at the $k$-th iteration,
\begin{equation} \label{eq:subp}
	V_k := \argmin_{V \in \cT_{X_k} \Odr} 
	h_k (V) := \jkh{\nabla \tilde{g} (X_k, \mu_k), V} + \dfrac{1}{2 \mu_k} \norm{V}\fs + s (X_k + V),
\end{equation}
where $X_k \in \Odr$ and $\mu_k > 0$ are the current iterate and smoothing parameter, respectively.
The above subproblem involves minimizing a strongly convex function on the tangent space. 
Although $V_k$ serves as a descent direction, the updated iterate $X_k + \alpha_k V_k$, for an arbitrary stepsize $\alpha_k > 0$, does not necessarily remain on $\Odr$. 
Consequently, we then perform a retraction to bring it back to $\Odr$.

Algorithm \ref{alg:SMPG} outlines the complete procedure of our approach for solving problem \eqref{opt:main}, which is named {\it smoothing manifold proximal gradient} and abbreviated to SMPG.
It is noteworthy that SMPG involves an Armijo line search procedure \eqref{eq:ls} to determine the stepsize. 
As we will show later, this backtracking line search procedure is well-defined and guaranteed to terminate in a finite number of steps.

\begin{algorithm2e}[ht]
	%\SetAlgoLined
	\caption{Smoothing manifold proximal gradient (SMPG).} 
	\label{alg:SMPG}
	
	\KwIn{$X_{0} \in \Odr$, $\mu_0 > 0$, $\bar{\mu} \in [0, \mu_0]$, $\theta \in (0, 1)$, and $\beta \in (0, 1)$.}
	
	\For{$k = 0, 1, 2, \dotsc$}{
		
		Solve subproblem \eqref{eq:subp} to obtain $V_k$.
		
		\If{$\norm{V_k}\ff \leq \bar{\mu}^{2}$ and $\mu_k \leq \bar{\mu}$}{
			
			Return $X_k$.
			
		}
		
		\Else{
			
			Find $\alpha_k := \beta^{m_k}$ such that $m_k$ is the smallest integer satisfying
			\begin{equation} \label{eq:ls}
				\tilde{f} (\retr_{X_k} (\beta^{m_k} V_k), \mu_k) 
				\leq \tilde{f} (X_k, \mu_k) - \dfrac{\beta^{m_k}}{2 \mu_k} \norm{V_k}\fs.
			\end{equation}
			
			Update $X_{k + 1} := \retr_{X_k} (\alpha_k V_k)$. 
			
			Set
			\begin{equation*}
				\mu_{k + 1} := 
				\left\{
				\begin{aligned}
					& \mu_{k}, 
					&& \mbox{if~} \norm{V_k}\ff > \mu_k^{2}, \\
					& \theta \mu_{k}, 
					&& \mbox{if~} \norm{V_k}\ff \leq \mu_k^{2}.
				\end{aligned}
				\right.
			\end{equation*}
			
		}
		
	}
	
	\KwOut{$X_k$.}
	
\end{algorithm2e}

\section{Convergence Analysis}

\label{sec:convergence}

This section delves into a comprehensive convergence analysis of the proposed algorithm.
Specifically, we establish that any accumulation point of the sequence generated by Algorithm \ref{alg:SMPG} is a stationary point.
And the iteration complexity of Algorithm \ref{alg:SMPG} is provided to attain an approximate stationary point.

\subsection{Descent Property}

In the following lemma, we first prove that $V_k$, obtained by solving subproblem \eqref{eq:subp}, serves as a descent direction in the tangent space $\cT_{X_k} \Odr$ at the current iterate $X_k \in \Odr$.

\begin{lemma} \label{le:des-h}
	Suppose that $\{X_k\}$ is the sequence generated by Algorithm \ref{alg:SMPG}. 
	Then it holds that
	\begin{equation*}
		h_k (0) - h_k (\alpha V_k)
		\geq \dfrac{\alpha (2- \alpha)}{2 \mu_k} \norm{V_k}\fs,
	\end{equation*}
	for any $\alpha \in [0, 1]$.
\end{lemma}

\begin{proof}
	Since $h_k (V)$ is strongly convex with the modulus $\mu_k^{-1}$, we have
	\begin{equation} \label{eq:h_k}
		h_k (\bar{V}) 
		\geq h_k (V)
		+ \jkh{\partial h_k (V), \bar{V} - V}
		+ \dfrac{1}{2 \mu_k} \norm{\bar{V} - V}\fs,
	\end{equation}
	for any $V, \bar{V} \in \Rdr$.
	In particular, for $V, \bar{V} \in \cT_{X_k} \Odr$, it holds that
	\begin{equation*}
		\jkh{\partial h_k (V), \bar{V} - V}
		= \jkh{\proj_{\cT_{X_k} \Odr} \dkh{\partial h_k (V)}, \bar{V} - V}.
	\end{equation*}
	Moreover, it follows from the optimality condition of subproblem \eqref{eq:subp} that $0 \in \proj_{\cT_{X_k} \Odr} \dkh{\partial h_k (V_k)}$.
	Taking $V = V_k$ and $\bar{V} = 0$ in \eqref{eq:h_k} yields that
	\begin{equation*}
		h_k (0) \geq h_k (V_k) + \dfrac{1}{2 \mu_k} \norm{V_k}\fs,
	\end{equation*}
	which, after a suitable rearrangement, can be equivalently written as
	\begin{equation*}
		s (X_k) 
		\geq \jkh{\nabla \tilde{g} (X_k, \mu_k), V_k}
		+ \dfrac{1}{\mu_k} \norm{V_k}\fs
		+ s (X_k + V_k).
	\end{equation*}
	According to the convexity of $s$, we have
	\begin{equation*}
		\begin{aligned}
			s (X_k + \alpha V_k) - s (X_k)
			= {} & s ( (1 - \alpha) X_k + \alpha (X_k + V_k) ) - s (X_k) \\
			\leq {} & \alpha \dkh{s (X_k + V_k) - s (X_k)}.
		\end{aligned}
	\end{equation*}
	Collecting the above two relationships together results in that
	\begin{equation*}
		\begin{aligned}
			h_k (\alpha V_k) - h_k (0)
			= {} & \alpha \jkh{\nabla \tilde{g} (X_k, \mu_k), V_k}
			+ \dfrac{\alpha^2}{2 \mu_k} \norm{V_k}\fs
			+ s (X_k + \alpha V_k) - s (X_k) \\
			\leq {} & \alpha \dkh{
				\jkh{\nabla \tilde{g} (X_k, \mu_k), V_k}
				+ \dfrac{\alpha}{2 \mu_k} \norm{V_k}\fs
				+ s (X_k + V_k) - s (X_k)} \\
			\leq {} & \dfrac{\alpha (\alpha - 2)}{2 \mu_k} \norm{V_k}\fs,
		\end{aligned}
	\end{equation*}
	which completes the proof.
\end{proof}

Based on Lemma \ref{le:des-h}, we can proceed to show that the line search procedure in Algorithm \ref{alg:SMPG} is well-defined, ensuring that the stepsize $\alpha_k$ can be determined in a finite number of trials.

\begin{lemma} \label{le:des-f}
	Let $\{X_k\}$ be the sequence generated by Algorithm \ref{alg:SMPG}.
	Then there exists a constant $\bar{\alpha} \in (0, 1]$ such that
	\begin{equation*}
		\tilde{f} (X_k, \mu_k) - \tilde{f} (\retr_{X_k} (\alpha V_k), \mu_k)
		\geq \dfrac{\alpha}{2 \mu_k} \norm{V_k}\fs,
	\end{equation*}
	for any $\alpha \in (0, \bar{\alpha})$.
\end{lemma}

\begin{proof}
	According to the Lipschitz continuity of $\nabla \tilde{g} (X, \mu)$ for fixed $\mu > 0$, it follows that
	\begin{equation*}
		\begin{aligned}
			\tilde{g} (\retr_{X_k} (\alpha V_k), \mu_k)
			\leq {} & \tilde{g} (X_k, \mu_k)
			+ \jkh{\nabla \tilde{g} (X_k, \mu_k), \retr_{X_k} (\alpha V_k) - X_k} \\
			& + \dfrac{L_u \mu_k + L_w}{2 \mu_k} \norm{\retr_{X_k} (\alpha V_k) - X_k}\fs \\
			\leq {} & \tilde{g} (X_k, \mu_k)
			+ \jkh{\nabla \tilde{g} (X_k, \mu_k), \retr_{X_k} (\alpha V_k) - (X_k + \alpha V_k)} \\
			& + \alpha \jkh{\nabla \tilde{g} (X_k, \mu_k), V_k}
			+ \dfrac{L_u \mu_0 + L_w}{2 \mu_k} \norm{\retr_{X_k} (\alpha V_k) - X_k}\fs,
		\end{aligned}
	\end{equation*}
	where the last inequality holds due to the fact that $\mu_k \leq \mu_0$.
	Since $\nabla u$ is continuous over the compact manifold $\Odr$, there exists a constant $M_u > 0$ such that $\norm{\nabla u (X)}\ff \leq M_u$ for any $X \in \Odr$.
	Hence, it holds that
	\begin{equation*}
		\norm{\nabla \tilde{g} (X_k, \mu_k)}\ff 
		\leq \norm{\nabla u (X_k)}\ff + \norm{\nabla \tilde{w} (X_k, \mu_k)}\ff
		%\leq {} & M_u + M_w
		\leq \dfrac{(M_u + M_w) \mu_0}{\mu_k}.
	\end{equation*}
	By invoking Lemma \ref{le:retr}, we can obtain that
	\begin{equation*}
		\begin{aligned}
			& \jkh{\nabla \tilde{g} (X_k, \mu_k), \retr_{X_k} (\alpha V_k) - (X_k + \alpha V_k)} \\
			\leq {} & \norm{\nabla \tilde{g} (X_k, \mu_k)}\ff \norm{\retr_{X_k} (\alpha V_k) - (X_k + \alpha V_k)}\ff \\
			\leq {} & \dfrac{\alpha^2 \mu_0 M_2 (M_u + M_w)}{\mu_k} \norm{V_k}\fs,
		\end{aligned}
	\end{equation*}
	and
	\begin{equation*}
		\norm{\retr_{X_k} (\alpha V_k) - X_k}\fs
		\leq \alpha ^2 M_1^2 \norm{V_k}\fs.
	\end{equation*}
	Let $C = 2 \mu_0 M_2 (M_u + M_w) + M_1^2 (L_u \mu_0 + L_w) > 0$ be a constant. 
	Then it can be readily verified that
	\begin{equation*}
		\tilde{g} (\retr_{X_k} (\alpha V_k), \mu_k)
		- \tilde{g} (X_k, \mu_k)
		\leq \alpha \jkh{\nabla \tilde{g} (X_k, \mu_k), V_k}
		+ \dfrac{\alpha^2 C}{2 \mu_k} \norm{V_k}\fs.
	\end{equation*}
	Moreover, according to the Lipschitz continuity of $s$, we have
	\begin{equation*}
		\begin{aligned}
			s (\retr_{X_k} (\alpha V_k)) - s (X_k)
			= {} & s (\retr_{X_k} (\alpha V_k)) - s (X_k + \alpha V_k) + s (X_k + \alpha V_k) - s (X_k) \\
			\leq {} & L_s \norm{\retr_{X_k} (\alpha V_k)) - (X_k + \alpha V_k)}\ff 
			+ s (X_k + \alpha V_k) - s (X_k) \\
			\leq {} & \dfrac{\alpha^2 \mu_0 M_2 L_s}{\mu_k} \norm{V_k}\fs 
			+ s (X_k + \alpha V_k) - s (X_k),
		\end{aligned}
	\end{equation*}
	where the last inequality results from Lemma \ref{le:retr} and the fact that $\mu_k \leq \mu_0$.
	Collecting the above two inequalities together yields that
	\begin{equation} \label{eq:des-f}
		\begin{aligned}
			& \tilde{f} (\retr_{X_k} (\alpha V_k), \mu_k) - \tilde{f} (X_k, \mu_k) \\
			= {} & \tilde{g} (\retr_{X_k} (\alpha V_k), \mu_k)
			- \tilde{g} (X_k, \mu_k)
			+ s (\retr_{X_k} (\alpha V_k)) 
			- s (X_k) \\
			\leq {} & \alpha \jkh{\nabla \tilde{g} (X_k, \mu_k), V_k}
			+ s (X_k + \alpha V_k) - s (X_k)
			+ \dfrac{\alpha^2 (C + 2 \mu_0 M_2 L_s)}{2 \mu_k} \norm{V_k}\fs.
		\end{aligned}
	\end{equation}
	As a direct consequence of Lemma \ref{le:des-h}, we can proceed to show that
	\begin{equation*}
		\alpha \jkh{\nabla \tilde{g} (X_k, \mu_k), V_k} 
		+ s (X_k + \alpha V_k) - s (X_k)
		\leq - \dfrac{\alpha}{\mu_k} \norm{V_k}\fs,
	\end{equation*}
	which together with the relationship \eqref{eq:des-f} implies that
	\begin{equation*}
		\tilde{f} (\retr_{X_k} (\alpha V_k), \mu_k) - \tilde{f} (X_k, \mu_k)
		\leq - \dfrac{\alpha}{2 \mu_k} \dkh{2 - \alpha \dkh{C + 2 \mu_0 M_2 L_s}} \norm{V_k}\fs.
	\end{equation*}
	Let $\bar{\alpha} = \min\{1, 1 / (C + 2 \mu_0 M_2 L_s)\} \in (0, 1]$. 
	Then for any $\alpha \in (0, \bar{\alpha})$, we can conclude that
	\begin{equation*}
		\tilde{f} (\retr_{X_k} (\alpha V_k), \mu_k) - \tilde{f} (X_k, \mu_k)
		\leq - \dfrac{\alpha}{2 \mu_k} \norm{V_k}\fs,
	\end{equation*}
	as desired. 
	The proof is completed.
\end{proof}

Lemma \ref{le:des-f} guarantees that the line search procedure in \eqref{eq:ls} terminates in at most $\lceil \log_{\beta} \bar{\alpha} \rceil$ steps, which is independent of the smoothing parameter $\mu_k$.
Here, the notation $\lceil m \rceil$ represents the smallest integer greater than or equal to $m \in \bR$.

\subsection{Global Convergence}

The following lemma lays the foundation for the global convergence analysis of Algorithm \ref{alg:SMPG} with $\bar{\mu} = 0$, as established in Theorem \ref{thm:global}.

\begin{lemma} \label{le:infinite}
	Let $\{X_k\}$ be the sequence generated by Algorithm \ref{alg:SMPG} with $\bar{\mu} = 0$.
	Then $\bK := \{k \in \bN \mid \norm{V_k}\ff \leq \mu_k^{2}\}$ is an infinite set.
\end{lemma}

\begin{proof}
	Suppose, on the contrary, that $\bK$ is a finite set.
	Then there exists $\bar{k} \in \bN$ such that
	\begin{equation} \label{eq:vk}
		\mu_k = \mu_{\bar{k}} > 0, 
		\mbox{~and~}
		\norm{V_k}\ff > \mu_{\bar{k}}^{2} > 0, 
	\end{equation}
	for any $k \geq \bar{k}$.
	Thus, we have $X_{k + 1} = \retr_{X_k}(\alpha_k V_k)$ for all $k \geq \bar{k}$, where the stepsize $\alpha_k$ is obtained by using the line search procedure in \eqref{eq:ls} with $\mu_k$ fixed as $\mu_{\bar{k}}$.
	From Lemma \ref{le:des-f}, we know that $\alpha_k \geq \bar{\alpha} \beta$ and
	\begin{equation*}
		\tilde{f} (X_k, \mu_{\bar{k}}) 
		- \tilde{f} (X_{k + 1}, \mu_{\bar{k}})
		\geq \dfrac{\alpha_k}{2 \mu_{\bar{k}}} \norm{V_k}\fs
		\geq \dfrac{\bar{\alpha} \beta}{2 \mu_{\bar{k}}} \norm{V_k}\fs,
	\end{equation*}
	for any $k \geq \bar{k}$.
	This observation indicates that the sequence $\{\tilde{f} (X_k, \mu_{\bar{k}})\}_{k \geq \bar{k}}$ is monotonically decreasing.
	Moreover, as a direct consequence of Proposition \ref{prop:smoothing}, we can proceed to show that
	\begin{equation*}
		f (X) \leq \tilde{f} (X, \mu) \leq f (X) + \dfrac{\mu \rho}{2},
	\end{equation*}
	for any $X \in\Odr$ and $\mu > 0$.
	In light of the continuity of $f$ over the compact manifold $\Odr$, there exist two constants $\tilde{f}_{\min}$ and $\tilde{f}_{\max}$ such that
	\begin{equation} \label{eq:bound-f}
		\tilde{f}_{\min} \leq \tilde{f} (X, \mu) \leq \tilde{f}_{\max},
	\end{equation} 
	for any $X \in\Odr$ and $\mu \in (0, \mu_0]$.
	Hence, the sequence $\{\tilde{f} (X_k, \mu_{\bar{k}})\}_{k \geq \bar{k}}$ is convergent.
	Then we can obtain that
	\begin{equation*}
		\lim_{k \to \infty} \norm{V_k}\fs
		\leq \dfrac{2 \mu_{\bar{k}}}{\bar{\alpha} \beta} \lim_{k \to \infty} \dkh{\tilde{f} (X_k, \mu_{\bar{k}}) - \tilde{f} (X_{k + 1}, \mu_{\bar{k}})}
		= 0,
	\end{equation*}
	which contradicts the second relationship in \eqref{eq:vk}.
	Consequently, we can conclude that $\bK$ is an infinite set.
	The proof is completed. 
\end{proof}

Now we are in the position to establish the global convergence of Algorithm \ref{alg:SMPG} to a stationary point of problem \eqref{opt:main} under the setting $\bar{\mu} = 0$.

\begin{theorem} \label{thm:global}
	Suppose that $\{X_k\}$ is the sequence generated by Algorithm \ref{alg:SMPG} with $\bar{\mu} = 0$.
	Then the sequence $\{X_k\}$ has at least one accumulation point.
	And any accumulation point is a stationary point of problem \eqref{opt:main}.
\end{theorem}

\begin{proof}
	For each $k \in \bK$, we have $\mu_{k + 1} = \theta \mu_k$ with $\theta \in (0, 1)$ being a decaying factor.
	According to Lemma \ref{le:infinite}, the index set $\bK$ is infinite.
	Then it can be readily verified that
	\begin{equation} \label{eq:lim-vk}
		\lim_{\bK \ni k \to \infty} \dfrac{1}{\mu_k} \norm{V_k}\ff 
		\leq \lim_{\bK \ni k \to \infty} \mu_k
		= 0.
	\end{equation}
	Since $\Odr$ is a compact manifold, the sequence $\{X_k\}$ is bounded.
	Then from the Bolzano-Weierstrass theorem, it can be deduced that the sequence $\{X_k\}$ has at least one accumulation point.
	Let $X\last$ be an accumulation point of $\{X_k\}$.
	The completeness of $\Odr$ guarantees that $X\last \in \Odr$. 
	By passing to a subsequence if necessary, we may assume without loss of generality that $\lim_{\bK \ni k \to \infty} X_k = X\last$. 
	%	\begin{equation*}
		%		\lim_{\bK \ni k \to \infty} X_k = X\last.
		%	\end{equation*}
	
	Next, by virtual of the optimality condition \cite{Yang2014optimality,Chen2020proximal} of subproblem \eqref{eq:subp}, there exists $H_k \in \partial s (X_k + V_k)$ such that
	\begin{equation} \label{eq:oc-subp}
		\grad\, u (X_k)
		+ \grad\, \tilde{w} (X_k, \mu_k)
		+ \proj_{\cT_{X_k} \Odr} \dkh{H_k} 
		+ \dfrac{1}{\mu_k} V_k
		= 0,
	\end{equation}
	for any $k \in \bK$.
	It is clear that the sequence $\{\grad\, u (X_k)\}$ is convergent and
	\begin{equation*}
		\lim_{\bK \ni k \to \infty} \grad\, u (X_k) = \grad\, u (X\last).
	\end{equation*}
	According to the Lipschitz continuity of $s$, the sequence $\{H_k\}$ is bounded \cite{Clarke1990optimization}.
	Without loss of generality, we assume that it is also convergent.
	Then there exists $H\last$ such that $\lim_{\bK \ni k \to \infty} H_k = H\last$. 
	%	\begin{equation*}
		%		\lim_{\bK \ni k \to \infty} H_k = H\last.
		%	\end{equation*}
	It follows from \cite[Theorem 24.4]{Rockafellar1970convex} that $H\last \in \partial s (X\last)$. 
	In addition, the boundedness of the sequence $\{\mu_k^{-1} V_k\}$ results from the fact that it is convergent in \eqref{eq:lim-vk}.
	As a result, the sequence $\{\grad\, \tilde{w} (X_k, \mu_k)\}$ is also bounded.
	Without loss of generality, we can assume that there exists $G\last$ such that
	\begin{equation*}
		\lim_{\bK \ni k \to \infty} \grad\, \tilde{w} (X_k, \mu_k) = G\last.
	\end{equation*}
	According to the definition of the Riemannian subdifferential of $w$ associated with $\tilde{w}$, it holds that $G\last \in \partial\lrmR |_{\tilde{w}} w (X\last)$.
	Then it follows from Theorem \ref{thm:consistency} that $G\last \in \partial\lrmR w (X\last)$.

	Finally, upon taking $k \to \infty$ in the relationship \eqref{eq:oc-subp}, we can conclude that
	\begin{equation*} %\label{eq:sc-rsubd}
		0 = \grad\, u (X\last)
		+ G\last
		+ \proj_{\cT_{X\last} \Odr} \dkh{H\last}
		\in \grad\, u (X\last)
		+ \partial\lrmR s (X\last)
		+ \partial\lrmR w (X\last),
	\end{equation*}
	which indicates that $X\last$ is a stationary point of problem \eqref{opt:main}. 
	We complete the proof. 
\end{proof}

\subsection{Iteration Complexity}

The final task is to derive the iteration complexity of Algorithm \ref{alg:SMPG}, a critical challenge that remains unresolved in existing works. 
The proof of Theorem \ref{thm:global} previously discussed leads to the insight that an accumulation point of $\{X_k\}$ is a stationary point of \eqref{opt:main} if the following conditions are satisfied,
\begin{equation*}
	\left\{
	\begin{aligned}
		& \lim_{k \to \infty} \dist \dkh{ 0, \grad\, u (X_k) + \grad\, \tilde{w} (X_k, \mu_k) + \partial\lrmR s (X_k + V_k) } = 0, \\
		& \lim_{k \to \infty} \norm{V_k}\ff = 0, \quad
		\lim_{k \to \infty} \mu_k = 0.
	\end{aligned}
	\right.
\end{equation*}
This observation motivates us to define the concept of $\epsilon$-approximate stationarity for problem \eqref{opt:main} as follows.

\begin{definition}
	A point $X \in \Odr$ is called an $\epsilon$-approximate stationary point of problem \eqref{opt:main} if there exists $V \in \cT_{X} \Odr$ with $\norm{V}\ff \leq \epsilon$ and $\mu \in [0, \epsilon]$ such that
	\begin{equation*}
		\dist \dkh{ 0, \grad\, u (X) + \grad\, \tilde{w} (X, \mu) + \proj_{\cT_{X} \Odr} \dkh{ \partial s (X + V) } } \leq \epsilon.
	\end{equation*}
\end{definition}

We show that Algorithm \ref{alg:SMPG} is capable of identifying an $\epsilon$-approximate stationary point of problem \eqref{opt:main} under the setting $\bar{\mu} = \epsilon$.

\begin{lemma} \label{le:iteration}
	For any $\epsilon \in (0, 1)$ and $\mu_0 \geq \epsilon$, Algorithm \ref{alg:SMPG} with $\bar{\mu} = \epsilon$ will terminate at an $\epsilon$-approximate stationary point of problem \eqref{opt:main}. 
\end{lemma}

\begin{proof}
	We first define the constant
	\begin{equation*}
		\iota_\epsilon = \left\lceil \log_{\theta} \dkh{ \epsilon / \mu_0 } \right\rceil + 1.
	\end{equation*}
	Let $\Bbbk_{i}$ be the $i$-th smallest number in $\bK$. 
	%For ease of notations, we define $\Bbbk_\epsilon := \Bbbk_{\iota_\epsilon}$. 
	Then it holds that
	\begin{equation*}
		\mu_{\Bbbk_{\iota_\epsilon}}
		= \theta^{\iota_\epsilon - 1} \mu_0 
		\leq \epsilon,
		\mbox{~and~}
		\norm{V_{\Bbbk_{\iota_\epsilon}}}\ff 
		\leq \mu_{\Bbbk_{\iota_\epsilon}}^{2}
		\leq \epsilon^{2},
	\end{equation*}
	which reveals that Algorithm \ref{alg:SMPG} will terminate at the iterate $X_{\Bbbk_{\iota_\epsilon}}$.
	
	The next step is to show that $X_{\Bbbk_{\iota_\epsilon}}$ is an $\epsilon$-approximate stationary point of problem \eqref{opt:main}. 
	According to the relationship \eqref{eq:oc-subp}, there exists $H_{\Bbbk_{\iota_\epsilon}} \in \partial s (X_{\Bbbk_{\iota_\epsilon}} + V_{\Bbbk_{\iota_\epsilon}})$ such that
	\begin{equation*}
		- \dfrac{1}{\mu_{\Bbbk_{\iota_\epsilon}}} V_{\Bbbk_{\iota_\epsilon}} 
		= \grad\, u (X_{\Bbbk_{\iota_\epsilon}}) 
		+ \grad\, \tilde{w} (X_{\Bbbk_{\iota_\epsilon}}, \mu_{\Bbbk_{\iota_\epsilon}})
		+ \proj_{\cT_{X_{\Bbbk_{\iota_\epsilon}}} \Odr} \dkh{ H_{\Bbbk_{\iota_\epsilon}} },
	\end{equation*}
	which implies that
	\begin{equation*}
		\begin{aligned}
			& \dist \dkh{ 0, \grad\, u (X_{\Bbbk_{\iota_\epsilon}}) + \grad\, \tilde{w} (X_{\Bbbk_{\iota_\epsilon}}, \mu_{\Bbbk_{\iota_\epsilon}}) + \proj_{\cT_{X_{\Bbbk_{\iota_\epsilon}}} \Odr} \dkh{ \partial s (X_{\Bbbk_{\iota_\epsilon}} + V_{\Bbbk_{\iota_\epsilon}}) } } \\
			\leq {} & \dfrac{1}{\mu_{\Bbbk_{\iota_\epsilon}}} \norm{V_{\Bbbk_{\iota_\epsilon}}}\ff
			\leq \mu_{\Bbbk_{\iota_\epsilon}}
			\leq \epsilon.
		\end{aligned}
	\end{equation*}
	Therefore, we conclude that $X_{\Bbbk_{\iota_\epsilon}}$ is an $\epsilon$-approximate stationary point of problem \eqref{opt:main}, which completes the proof. 
\end{proof}

The iteration complexity of Algorithm \ref{alg:SMPG} is established in the following theorem for finding an $\epsilon$-approximate stationary point.

\begin{theorem}
	\label{thm:complexity}
	For any $\epsilon \in (0, 1)$ and $\mu_0 \geq \epsilon$, Algorithm \ref{alg:SMPG} with $\bar{\mu} = \epsilon$ will reach an $\epsilon$-approximate stationary point of problem \eqref{opt:main} after at most $O (\epsilon^{-3})$ iterations. 
\end{theorem}

\begin{proof}
	According to Lemma \ref{le:iteration}, Algorithm \ref{alg:SMPG} with $\bar{\mu} = \epsilon$ will terminate at $X_{\Bbbk_{\iota_\epsilon}}$, which is an $\epsilon$-approximate stationary point of problem \eqref{opt:main}. 
	We give an upper bound of the iteration number $\Bbbk_{\iota_\epsilon}$.
	For convenience, we denote $\Bbbk_0 = 0$.
	The iterations from $\Bbbk_{i}$ to $\Bbbk_{i + 1}$ solve subproblem \eqref{eq:subp} with the smoothing parameter fixed as $\mu_{\Bbbk_{i + 1}}$ for $0 \leq i \leq \iota_\epsilon - 1$.
	According to Lemma \ref{le:des-f}, we have
	\begin{equation*}
		\tilde{f} (X_k, \mu_{\Bbbk_{i + 1}}) 
		- \tilde{f} (X_{k + 1}, \mu_{\Bbbk_{i + 1}})
		\geq \dfrac{\alpha_k}{2 \mu_{\Bbbk_{i + 1}}} \norm{V_k}\fs
		\geq \dfrac{\bar{\alpha} \beta}{2 \mu_{\Bbbk_{i + 1}}} \norm{V_k}\fs,
	\end{equation*}
	for $\Bbbk_{i} \leq k \leq \Bbbk_{i + 1} - 1$.
	Then it can be readily verified that
	\begin{equation*}
		\begin{aligned}
			\sum_{k = \Bbbk_{i}}^{\Bbbk_{i + 1} - 1} \norm{V_k}\fs
			\leq {} & \dfrac{2 \mu_{\Bbbk_{i + 1}}}{\bar{\alpha} \beta} \sum_{k = \Bbbk_{i}}^{\Bbbk_{i + 1} - 1} \dkh{\tilde{f} (X_k, \mu_{\Bbbk_{i + 1}}) - \tilde{f} (X_{k + 1}, \mu_{\Bbbk_{i + 1}})} \\
			= {} & \dfrac{2 \mu_{\Bbbk_{i + 1}}}{\bar{\alpha} \beta} \dkh{\tilde{f} (X_{\Bbbk_{i}}, \mu_{\Bbbk_{i + 1}}) - \tilde{f} (X_{\Bbbk_{i + 1}}, \mu_{\Bbbk_{i + 1}})} \\
			\leq {} & \dfrac{2 \mu_{\Bbbk_{i + 1}}}{\bar{\alpha} \beta} \dkh{\tilde{f}_{\max} - \tilde{f}_{\min}},
		\end{aligned}
	\end{equation*}
	where the last inequality follows from \eqref{eq:bound-f}. 
	From the definition of $\Bbbk_{i + 1}$, we know that $\norm{V_k}\ff > \mu_{\Bbbk_{i + 1}}^{2}$ for $\Bbbk_{i} \leq k \leq \Bbbk_{i + 1} - 1$. 
	Hence, it holds that
	\begin{equation*}
		\sum_{k = \Bbbk_{i}}^{\Bbbk_{i + 1} - 1} \norm{V_k}\fs
		\geq \dkh{\Bbbk_{i + 1} - \Bbbk_{i}} \mu_{\Bbbk_{i + 1}}^{4},
	\end{equation*}
	which together with the relationship $\mu_{\Bbbk_{i + 1}} = \theta^i \mu_0$ implies that
	\begin{equation*}
		\Bbbk_{i + 1} - \Bbbk_{i}
		\leq \dfrac{2 (\tilde{f}_{\max} - \tilde{f}_{\min})}{\bar{\alpha} \beta \mu_{\Bbbk_{i + 1}}^{3}}
		= \dfrac{2 (\tilde{f}_{\max} - \tilde{f}_{\min})}{\bar{\alpha} \beta \mu_0^{3} \theta^{3 i}}.
	\end{equation*}
	Finally, a straightforward verification reveals that
	\begin{equation*}
		\Bbbk_{\iota_\epsilon}
		= \Bbbk_0 + \sum_{i = 0}^{\iota_\epsilon - 1} \dkh{\Bbbk_{i + 1} - \Bbbk_{i}}
		\leq \dfrac{2 (\tilde{f}_{\max} - \tilde{f}_{\min})}{\bar{\alpha} \beta \mu_0^{3}} \sum_{i = 0}^{\iota_\epsilon - 1} \dfrac{1}{\theta^{3 i}}
		\leq \dfrac{2 \theta^{3} (\tilde{f}_{\max} - \tilde{f}_{\min})}{\bar{\alpha} \beta (1 - \theta^{3}) \epsilon^{3}}.
	\end{equation*}
	The proof is completed. 
\end{proof}

Theorem \ref{thm:complexity} demonstrates that SMPG achieves an iteration complexity of $O (\epsilon^{-3})$. 
By contrast, the Riemannian subgradient method \cite{Li2021weakly}, which is capable of handling problem \eqref{opt:main}, suffers from an inferior iteration complexity of $O (\epsilon^{-4})$.

\section{Numerical Experiments}

\label{sec:experiment}

Preliminary numerical results are presented in this section to provide additional insights into the performance guarantees of model \eqref{opt:drspca-was} and Algorithm \ref{alg:SMPG} (SMPG). 
All codes are implemented in MATLAB R2018b on a workstation with dual Intel Xeon Gold 6242R CPU processors (at $3.10$ GHz$\times 20 \times 2$) and $510$ GB of RAM under Ubuntu 20.04.

\subsection{Experimental Setting}

In the following experiments, we estimate an empirical distribution \cite{Esfahani2018data} to serve as the nominal distribution $\dP\lcirc$ in problem \eqref{opt:drspca-was}. 
Although the true distribution $\dP\last$ remains inherently elusive, it is often partially observable through a finite collection of $n \in \bN$ independent samples \cite{Esfahani2018data,Lu2012augmented}, such as past realizations of the random vector $\xi$. 
Let the training dataset comprising these samples be denoted as $\hat{\Xi}_n := \{\hat{\xi}_i\}_{i = 1}^n \subseteq \Rd$. 
Then we can construct the empirical distribution as follows,
\begin{equation*}
	\hat{\dP}_n := \dfrac{1}{n} \sum_{i = 1}^{n} \eth_{\hat{\xi}_i},
\end{equation*}
where $\eth_{\hat{\xi}_i}$ represents the Dirac distribution concentrating unit mass at $\hat{\xi}_i \in \Rd$. 
In fact, the empirical distribution $\hat{\dP}_n$ can be interpreted as the uniform distribution over the finite samples in $\hat{\Xi}_n$.

Based on the preceding constructions, the sample average approximation (SAA) model of PCA can be expressed as
\begin{equation}
	\label{opt:spca-saa}
	\min_{X \in \Odr}
	\dE_{\hat{\dP}_n} \fkh{ \norm{ \dkh{I_d - X X\zz} \dkh{\xi -  \dE_{\hat{\dP}_n} \fkh{\xi} } }_2^2} + s (X).
	%\tr \dkh{ \dkh{I_d - X X\zz} \hat{\Sigma}_n }
\end{equation}
The above formulation has been extensively investigated in the literature \cite{Lu2012augmented,Wang2024decentralized,Wang2023communication,Wang2024seeking}, which does not account for uncertainty in the underlying distribution. 
In the subsequent experiments, we will conduct a performance comparison between the equivalent reformulation \eqref{opt:was-eq-r} of the DRO model \eqref{opt:drspca-was} and the SAA model \eqref{opt:spca-saa}.

In addition, we focus on the $\ell_1$-norm regularizer with a parameter $\gamma > 0$ to control the amount of sparseness, namely,
\begin{equation*}
	s (X) = \gamma \norm{X}_1,
\end{equation*}
where the $\ell_1$-norm of $X$ is given by $\norm{X}_1 := \sum_{i, j} \abs{X_{i, j}}$ with $X_{i, j}$ being the $(i, j)$-th entry of $X$. 
Our empirical experiments reveal that the choice of regularizers does not affect the numerical results dramatically.

\subsection{Performance of SMPG}

The first experiment is designed to demonstrate the effectiveness and efficiency of SMPG for solving problem \eqref{opt:was-eq-r} in comparison with the Riemannian subgradient method (RSM) proposed in \cite{Li2021weakly}. 
Specifically, we construct the true distribution $\dP\last$ based on the normal distribution $\dN (0, \Sigma\last)$. 
The true covariance matrix $\Sigma\last \in \bS_{+}^d$ is obtained by projecting a randomly generated matrix in $\Rdd$ onto $\bS_{+}^d$. 
Subsequently, we produce $n$ samples independently and identically from $\dN (0, \Sigma\last)$ to generate the empirical distribution $\hat{\dP}_n$.

For our testing, we fix $n = 50$, $r = 50$, $\gamma = 0.05$, and $\rho = 1$ in problem \eqref{opt:was-eq-r}. 
The algorithmic parameters of SMPG are set to $\mu_0 = 0.1$, $\theta = 0.5$, $\bar{\mu} = 0$, and $\beta = 0.5$. 
And RSM is equipped with the diminishing stepsize $5 / \sqrt{k}$ for each iteration $k$. 
Moreover, we construct the initial point based on the leading $r$ eigenvectors of the empirical covariance matrix. 
The fixed-point method proposed in \cite{Liu2024penalty} is employed to solve the subproblem \eqref{eq:subp}, and the retraction operator is realized by the polar decomposition. 
Finally, we terminate SMPG and RSM after $1000$ and $3000$ iterations, respectively.

Figure \ref{fig:smpg} comprises two subplots that depict CPU times and final function values obtained by the two algorithms for the problem dimension $d$ varying across $\{1000, 1500, 2000, 2500, 3000\}$. 
It can be observed that the proposed SMPG algorithm consistently yields solutions of higher quality, as evidenced by its lower function values. 
Furthermore, with the exception of the case $d =1000$, SMPG outperforms RSM in terms of computational efficiency, requiring significantly less CPU times. 
Notably, the performance advantage of SMPG becomes increasingly pronounced as the problem dimension grows.

\begin{figure}[t]
	\centering
	\subfigure[Function Value]{
		\label{subfig:smpg_fval}
		\includegraphics[width=0.3\linewidth]{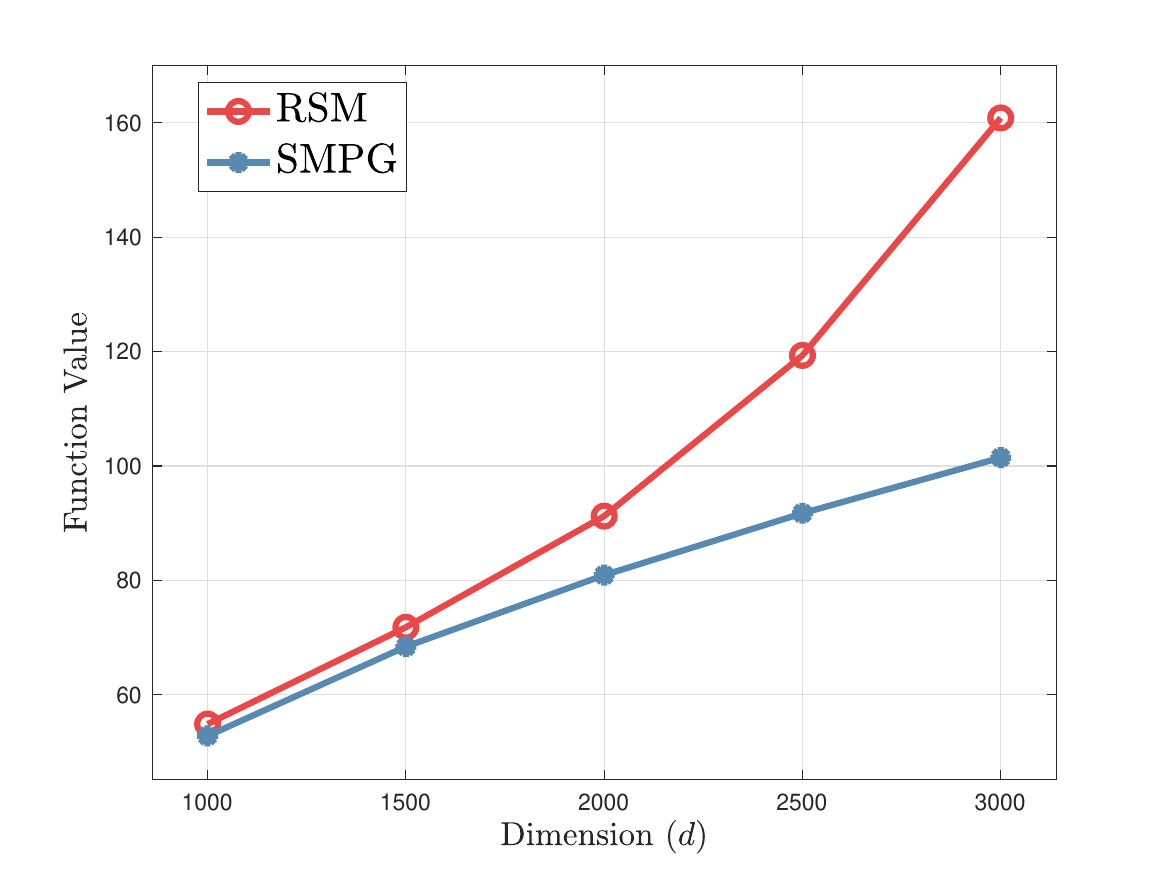}
	}
	\subfigure[CPU Time]{
		\label{subfig:smpg_time}
		\includegraphics[width=0.3\linewidth]{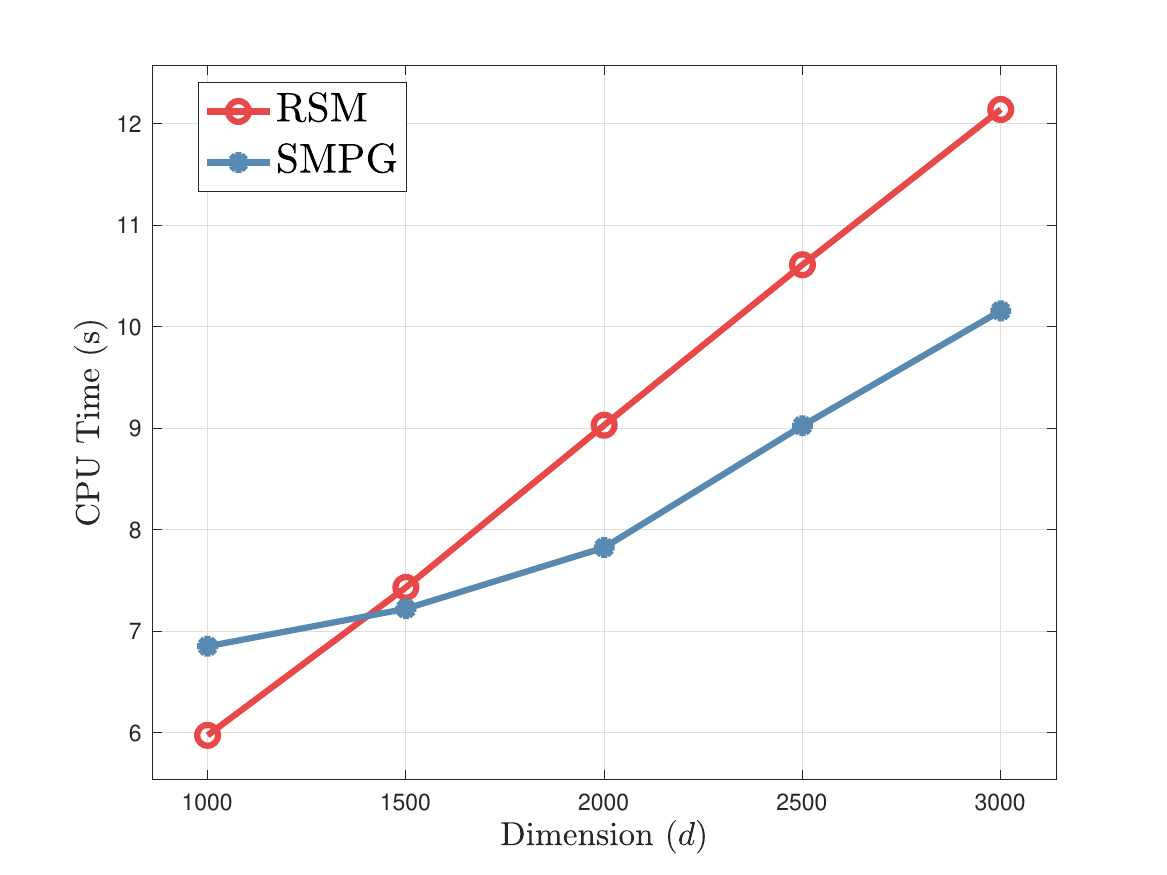}
	}
	\caption{Numerical comparison between SMPG and RSM for different problem dimensions.}
	\label{fig:smpg}
\end{figure}

\subsection{Performance of DRO Model}

In the next experiment, we aim to illustrate the rationality and necessity of adopting the DRO model for PCA. 
For convenience, the DRO model \eqref{opt:was-eq-r} and the SAA model \eqref{opt:spca-saa} are denoted by DRPCA and PCA-SAA, respectively.
The performances of DRPCA and PCA-SAA are evaluated on three real-world datasets, including MNIST\footnote{\url{https://yann.lecun.com/exdb/mnist/}}, CIFAR-10\footnote{\url{https://www.cs.toronto.edu/~kriz/cifar.html}}, and WHO-Mortality\footnote{\url{https://www.who.int/data/gho/data/themes/mortality-and-global-health-estimates}}. 
Specifically, MNIST contains $60000$ samples, each with $d = 784$ features; CIFAR-10 consists of $50000$ samples with $d = 3072$ features per sample; and WHO-Mortality includes $6080$ samples, each characterized by $d = 24$ features.

For each dataset, we extract the first $n$ samples to construct the empirical distribution $\hat{\dP}_n$. 
Then SMPG and ManPG \cite{Chen2020proximal} are deployed to solve the DRPCA model \eqref{opt:was-eq-r} and the PCA-SAA model \eqref{opt:spca-saa}, respectively. 
For our simulation in this case, we fix $r = 5$ and $\gamma = 0.02$ in both problems \eqref{opt:was-eq-r} and \eqref{opt:spca-saa}.

The performance of DRPCA and PCA-SAA is first evaluated under the worst-case scenario. 
In this test, we set the radius $\rho$ to $0.5$ in problem \eqref{opt:was-eq-r}. 
And the worst-case performance of solutions is represented by the objective function value of problem \eqref{opt:was-eq-r}. 
The corresponding numerical results are presented in Figure \ref{fig:drspca_worst} for varying sample sizes $n \in \{100, 200, 300, 400, 500\}$. 
Next, we assess the quality of solutions based on the following out-of-sample performance \cite{Esfahani2018data},
\begin{equation*}
	f\last (X) 
	= \dE_{\dP\last} \fkh{ \norm{ \dkh{I_d - X X\zz} \dkh{\xi -  \dE_{\dP\last} \fkh{\xi} } }_2^2} + s (X),
\end{equation*}
which is the objective function value of PCA with the true distribution $\dP\last$ being the empirical distribution generated by all samples in the dataset. 
In addition, the radius $\rho$ in problem \eqref{opt:was-eq-r} is set to $5 n^{-1/2}$ for each sample size $n$. 
Figure \ref{fig:drspca_true} visualizes the out-of-sample performances of DRPCA and PCA-SAA on two datasets, evaluated across sample sizes $n \in \{100, 200, 300, 400, 500\}$. 
It can be observed from Figure \ref{fig:drspca_worst} and Figure \ref{fig:drspca_true} that the solutions of DRPCA consistently demonstrate superior performances compared to those of PCA-SAA across all tested cases. 
These numerical results highlight the rationality and necessity of adopting the DRO model for PCA.

\begin{figure}[t]
	\centering
	\subfigure[MNIST]{
		\label{subfig:mnist_worst}
		\includegraphics[width=0.3\linewidth]{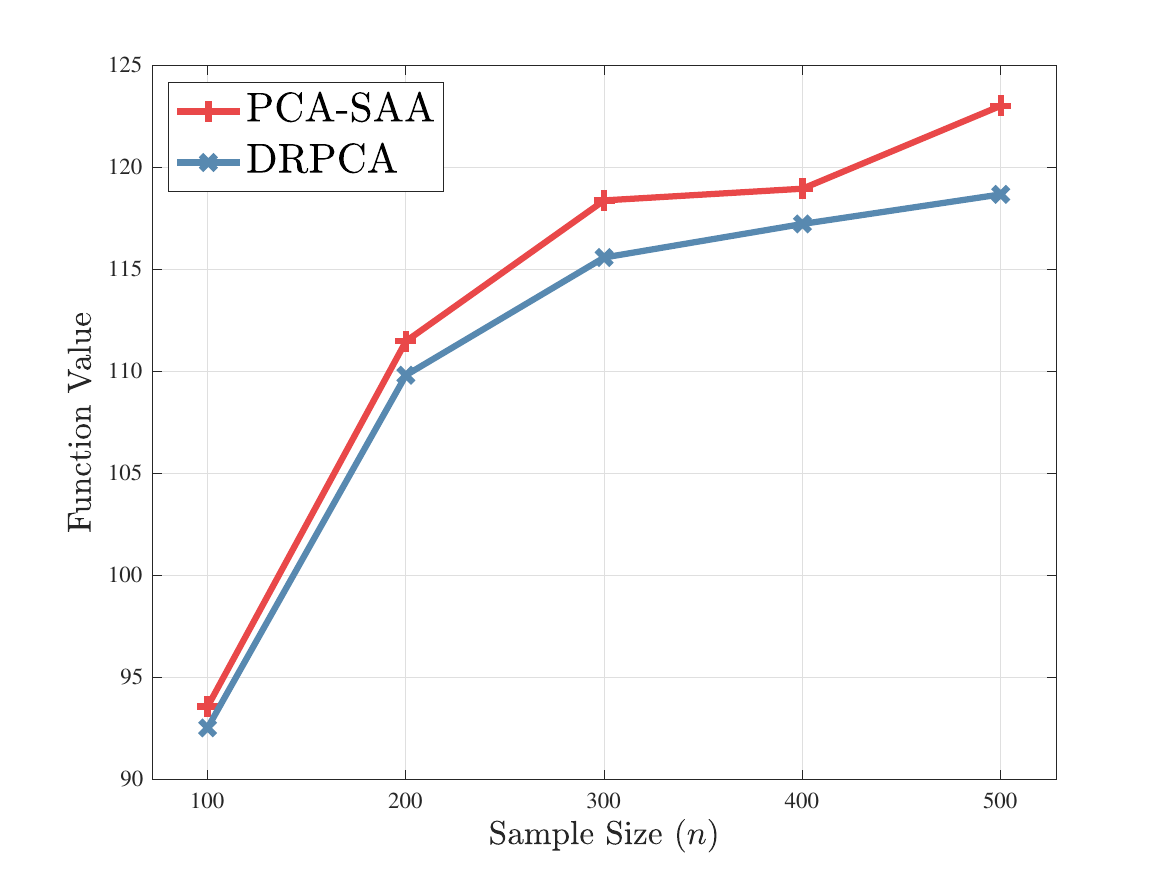}
	}
	\subfigure[CIFAR-10]{
		\label{subfig:cifar_worst}
		\includegraphics[width=0.3\linewidth]{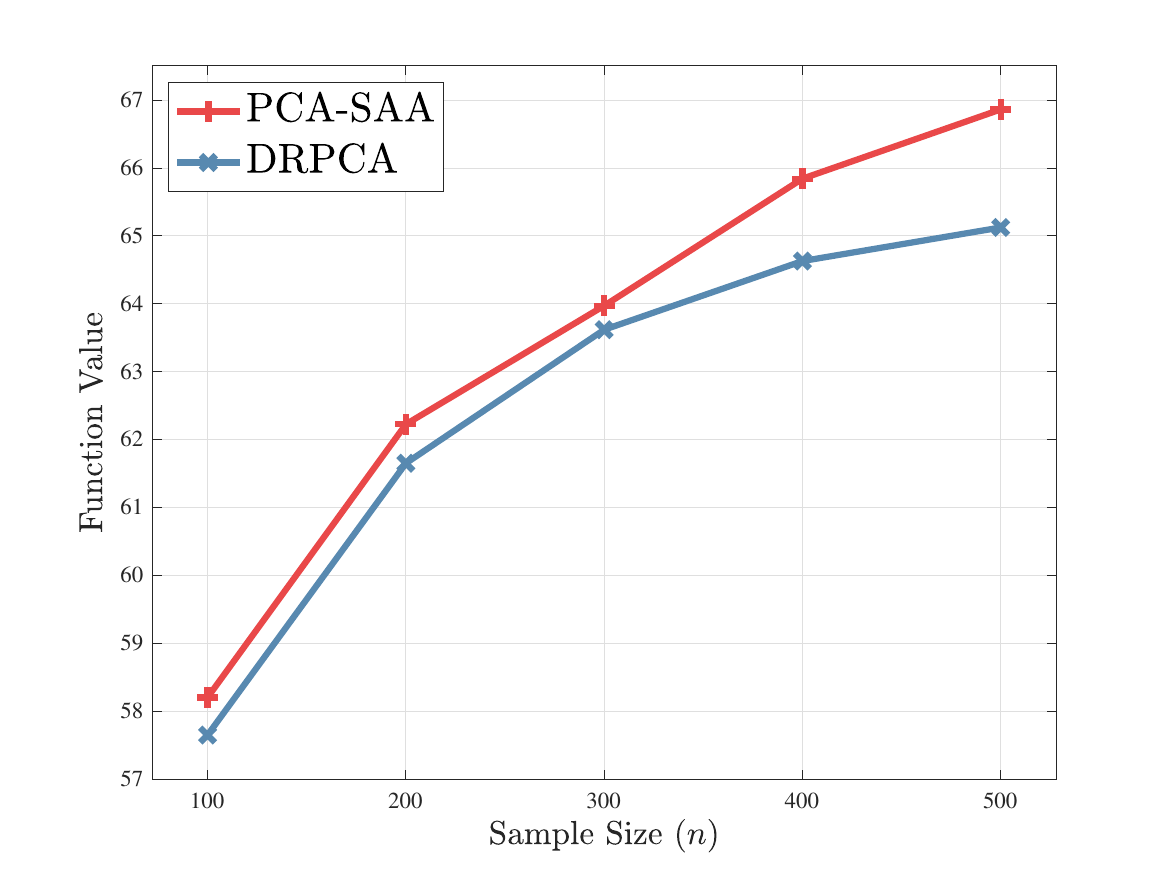}
	}
	\subfigure[WHO-Mortality]{
		\label{subfig:who_worst}
		\includegraphics[width=0.3\linewidth]{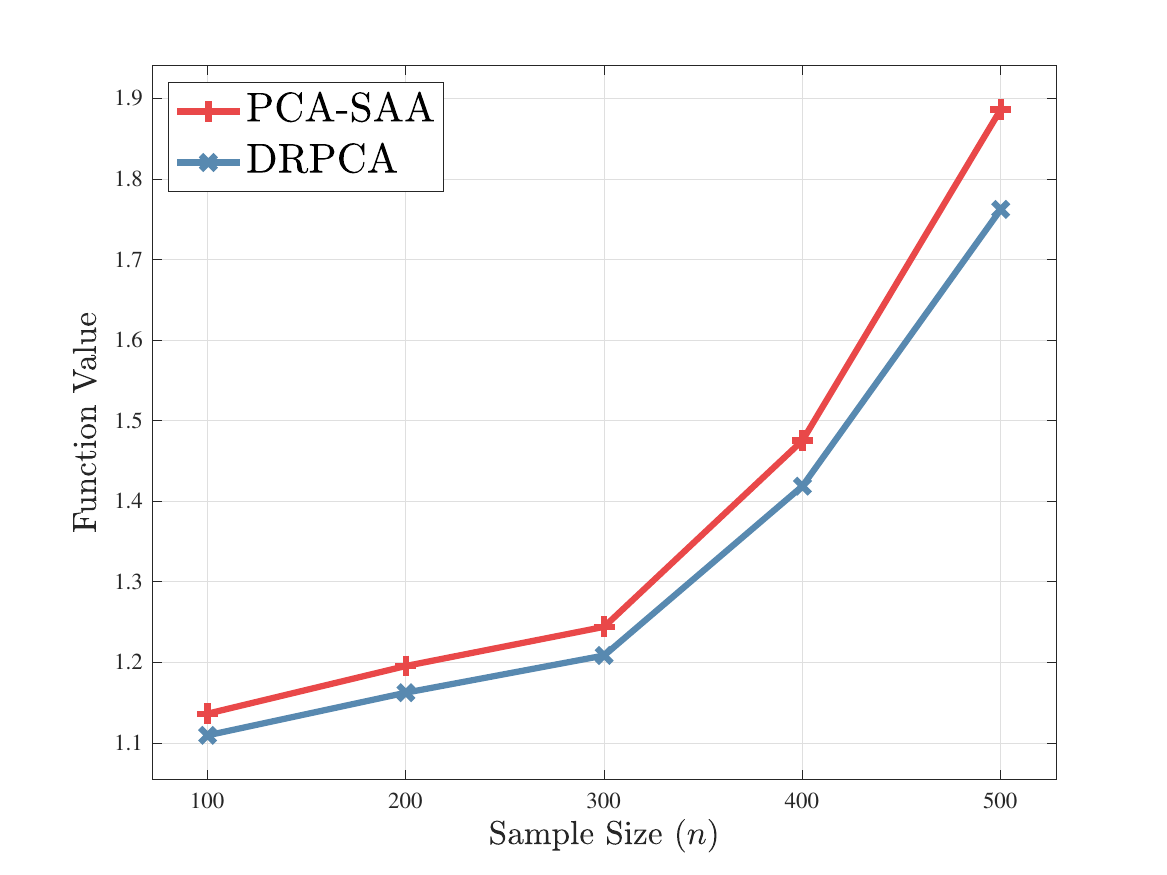}
	}
	\caption{Numerical comparison of the worst-case performance between DRPCA and PCA-SAA on three datasets.}
	\label{fig:drspca_worst}
\end{figure}

\begin{figure}[t]
	\centering
	\subfigure[MNIST]{
		\label{subfig:mnist_true}
		\includegraphics[width=0.3\linewidth]{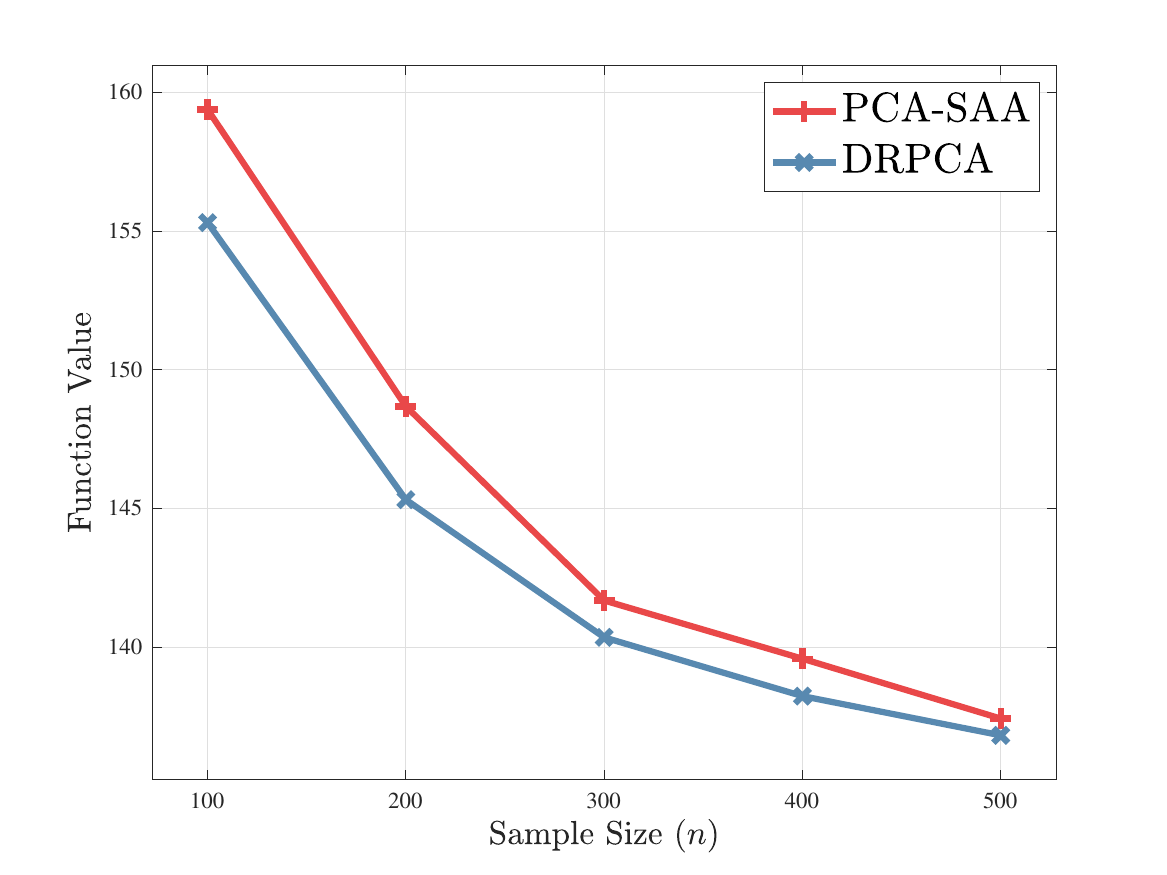}
	}
	\subfigure[CIFAR-10]{
		\label{subfig:cifar_true}
		\includegraphics[width=0.3\linewidth]{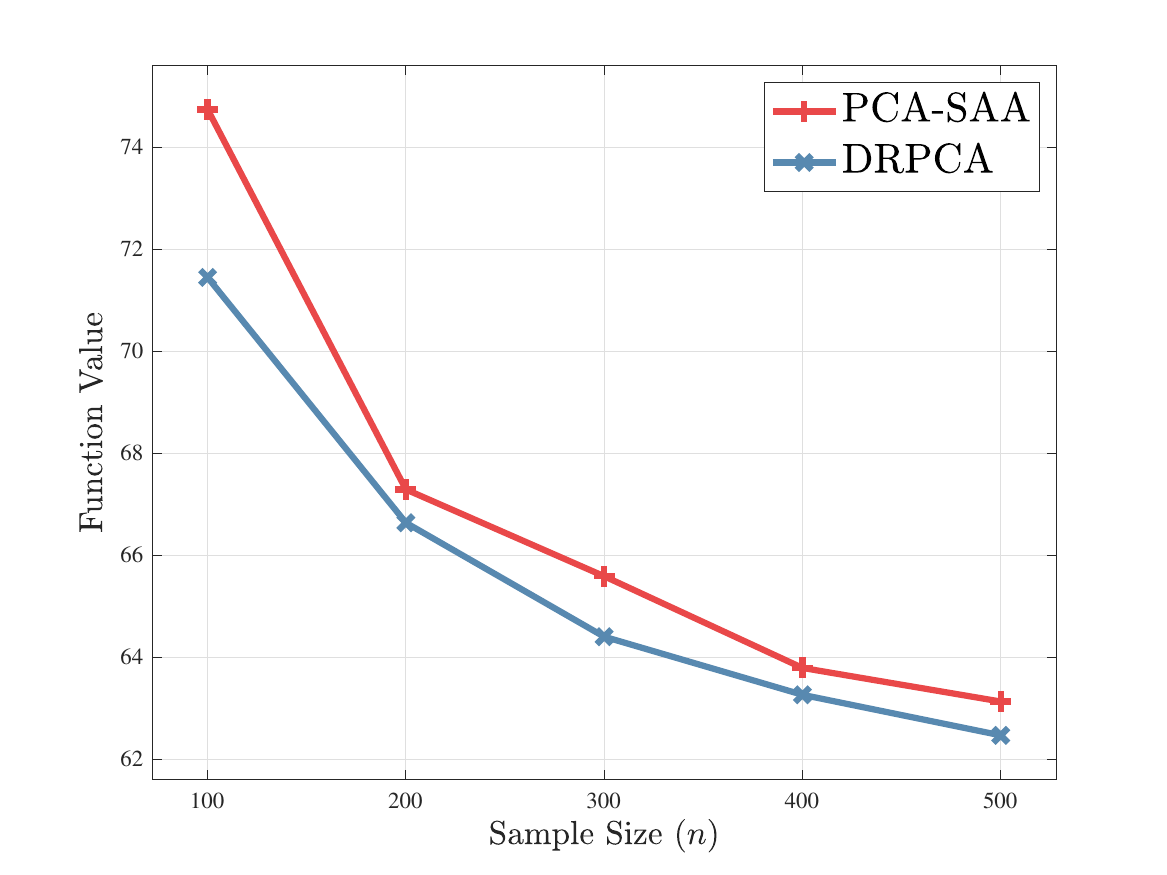}
	}
	\subfigure[WHO-Mortality]{
		\label{subfig:who_true}
		\includegraphics[width=0.3\linewidth]{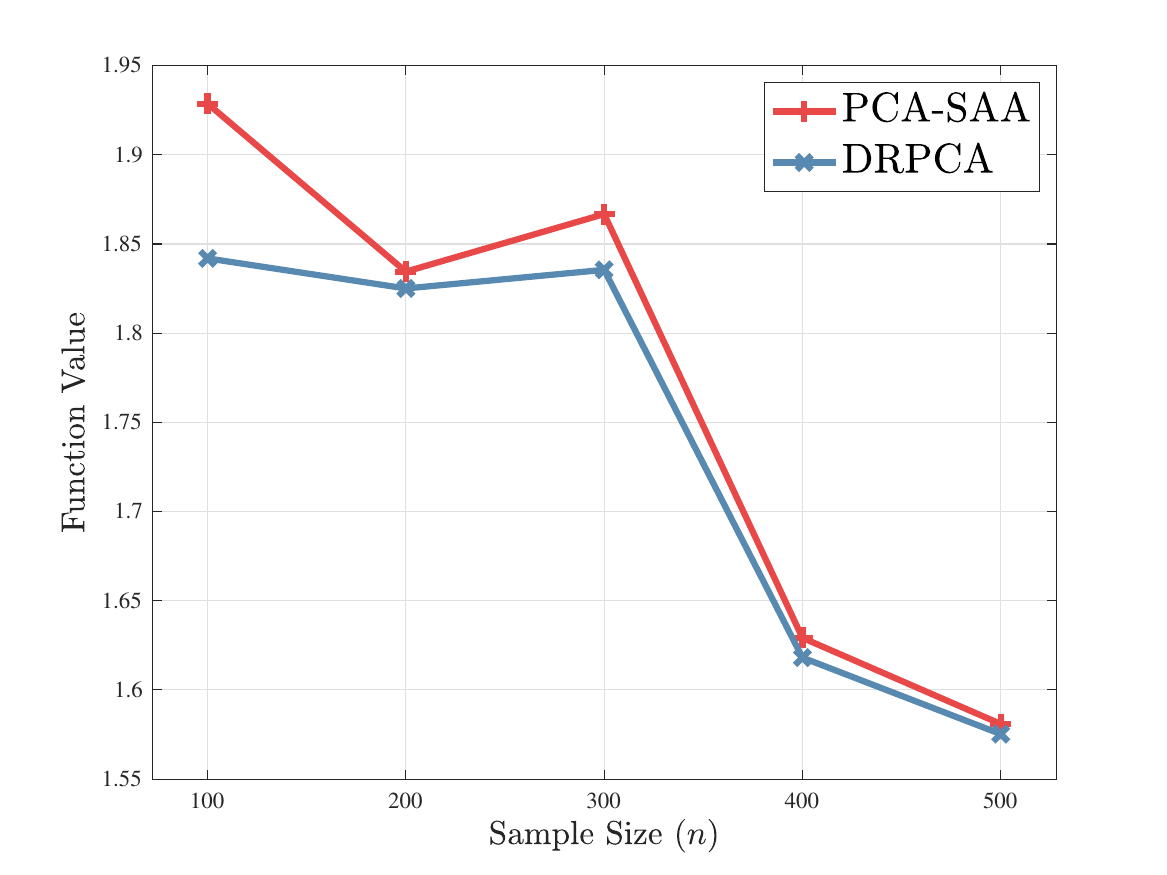}
	}
	\caption{Numerical comparison of the out-of-sample performance between DRPCA and PCA-SAA on three datasets.}
	\label{fig:drspca_true}
\end{figure}

\section{Concluding Remarks} 

\label{sec:conclusion}

The DRO model \eqref{opt:drspca-was} of PCA constitutes a nonsmooth constrained min-max optimization problem on a Riemannian manifold. 
When the ambiguity set is characterized by the type-$2$ Wasserstein distance, we equivalently reformulate it as a minimization problem \eqref{opt:was-eq-r} by providing a closed-form expression for the optimal value of the inner maximization problem in \eqref{opt:drspca-was}. 
However, problem \eqref{opt:was-eq-r} can hardly be solved efficiently by existing Riemannian optimization algorithms due to the involvement of two nonsmooth terms in the objective function. 
To surmount this issue, we develop an efficient algorithm SMPG for problem \eqref{opt:was-eq-r}, which incorporates the smoothing approximation technique into the proximal gradient method on Riemannian manifolds. 

We rigorously demonstrate that SMPG achieves the global convergence to a stationary point and further provide an iteration complexity. 
Preliminary numerical results are presented to validate the efficiency of SMPG and the effectiveness of our DRO model, illuminating their potential in addressing the challenges inherent in PCA under distributional uncertainty.

\section*{Acknowledgments}

We sincerely express our gratitude to Wei Bian, Hailin Sun, Nachuan Xiao, and Zaikun Zhang for their insightful discussions on smoothing algorithms, distributionally robust optimization, and manifold optimization.

% ---------------------------------------------------------------------------------------------------------------------------------

\bibliographystyle{abbrv}
\bibliography{library}

\begin{thebibliography}{10}

\bibitem{Absil2008optimization}
P.-A. Absil, R.~Mahony, and R.~Sepulchre.
\newblock {\em Optimization Algorithms on Matrix Manifolds}.
\newblock Princeton University Press, 2008.

\bibitem{Bento2017iteration}
G.~C. Bento, O.~P. Ferreira, and J.~G. Melo.
\newblock Iteration-complexity of gradient, subgradient and proximal point
  methods on {R}iemannian manifolds.
\newblock {\em Journal of Optimization Theory and Applications},
  173(2):548--562, 2017.

\bibitem{Blanchet2021sample}
J.~Blanchet and Y.~Kang.
\newblock Sample out-of-sample inference based on {W}asserstein distance.
\newblock {\em Operations Research}, 69(3):985--1013, 2021.

\bibitem{Blanchet2019robust}
J.~Blanchet, Y.~Kang, and K.~Murthy.
\newblock Robust {W}asserstein profile inference and applications to machine
  learning.
\newblock {\em Journal of Applied Probability}, 56(3):830--857, 2019.

\bibitem{Boumal2023introduction}
N.~Boumal.
\newblock {\em An Introduction to Optimization on Smooth Manifolds}.
\newblock Cambridge University Press, 2023.

\bibitem{Boumal2018global}
N.~Boumal, P.-A. Absil, and C.~Cartis.
\newblock Global rates of convergence for nonconvex optimization on manifolds.
\newblock {\em IMA Journal of Numerical Analysis}, 39(1):1--33, 2018.

\bibitem{Chen2021manifold}
S.~Chen, Z.~Deng, S.~Ma, and A.~M.-C. So.
\newblock Manifold proximal point algorithms for dual principal component
  pursuit and orthogonal dictionary learning.
\newblock {\em IEEE Transactions on Signal Processing}, 69:4759--4773, 2021.

\bibitem{Chen2020proximal}
S.~Chen, S.~Ma, A.~M.-C. So, and T.~Zhang.
\newblock Proximal gradient method for nonsmooth optimization over the
  {S}tiefel manifold.
\newblock {\em SIAM Journal on Optimization}, 30(1):210--239, 2020.

\bibitem{Chen2024nonsmooth}
S.~Chen, S.~Ma, A.~M.-C. So, and T.~Zhang.
\newblock Nonsmooth optimization over the {S}tiefel manifold and beyond:
  Proximal gradient method and recent variants.
\newblock {\em SIAM Review}, 66(2):319--352, 2024.

\bibitem{Chen2012smoothing}
X.~Chen.
\newblock Smoothing methods for nonsmooth, nonconvex minimization.
\newblock {\em Mathematical Programming}, 134:71--99, 2012.

\bibitem{Chen2025tight}
X.~Chen, Y.~He, and Z.~Zhang.
\newblock Tight error bounds for the sign-constrained {S}tiefel manifold.
\newblock {\em SIAM Journal on Optimization}, 35(1):302--329, 2025.

\bibitem{Chen2019discrete}
X.~Chen, H.~Sun, and H.~Xu.
\newblock Discrete approximation of two-stage stochastic and distributionally
  robust linear complementarity problems.
\newblock {\em Mathematical Programming}, 177:255--289, 2019.

\bibitem{Chu2024wasserstein}
H.~T.~M. Chu, M.~Lin, and K.-C. Toh.
\newblock {W}asserstein distributionally robust optimization and its tractable
  regularization formulations.
\newblock {\em arXiv:2402.03942}, 2024.

\bibitem{Clarke1990optimization}
F.~H. Clarke.
\newblock {\em Optimization and Nonsmooth Analysis}.
\newblock Society for Industrial and Applied Mathematics, Philadelphia, 1990.

\bibitem{Delage2010distributionally}
E.~Delage and Y.~Ye.
\newblock Distributionally robust optimization under moment uncertainty with
  application to data-driven problems.
\newblock {\em Operations Research}, 58(3):595--612, 2010.

\bibitem{Drusvyatskiy2019efficiency}
D.~Drusvyatskiy and C.~Paquette.
\newblock Efficiency of minimizing compositions of convex functions and smooth
  maps.
\newblock {\em Mathematical Programming}, 178:503--558, 2019.

\bibitem{Esfahani2018data}
P.~M. Esfahani and D.~Kuhn.
\newblock Data-driven distributionally robust optimization using the
  {W}asserstein metric: Performance guarantees and tractable reformulations.
\newblock {\em Mathematical Programming}, 171(1):115--166, 2018.

\bibitem{Gao2018new}
B.~Gao, X.~Liu, X.~Chen, and Y.-X. Yuan.
\newblock A new first-order algorithmic framework for optimization problems
  with orthogonality constraints.
\newblock {\em SIAM Journal on Optimization}, 28(1):302--332, 2018.

\bibitem{Gao2023finite}
R.~Gao.
\newblock Finite-sample guarantees for {W}asserstein distributionally robust
  optimization: Breaking the curse of dimensionality.
\newblock {\em Operations Research}, 71(6):2291--2306, 2023.

\bibitem{Gao2023distributionally}
R.~Gao and A.~Kleywegt.
\newblock Distributionally robust stochastic optimization with {W}asserstein
  distance.
\newblock {\em Mathematics of Operations Research}, 48(2):603--655, 2023.

\bibitem{Gelbrich1990formula}
M.~Gelbrich.
\newblock On a formula for the ${L}^2$ {W}asserstein metric between measures on
  {E}uclidean and {H}ilbert spaces.
\newblock {\em Mathematische Nachrichten}, 147(1):185--203, 1990.

\bibitem{Goh2010distributionally}
J.~Goh and M.~Sim.
\newblock Distributionally robust optimization and its tractable
  approximations.
\newblock {\em Operations Research}, 58(4-part-1):902--917, 2010.

\bibitem{Higham2008functions}
N.~J. Higham.
\newblock {\em Functions of Matrices: Theory and Computation}.
\newblock Society for Industrial and Applied Mathematics, Philadelphia, 2008.

\bibitem{Hosseini2018line}
S.~Hosseini, W.~Huang, and R.~Yousefpour.
\newblock Line search algorithms for locally {L}ipschitz functions on
  {R}iemannian manifolds.
\newblock {\em SIAM Journal on Optimization}, 28(1):596--619, 2018.

\bibitem{Hosseini2017riemannian}
S.~Hosseini and A.~Uschmajew.
\newblock A {R}iemannian gradient sampling algorithm for nonsmooth optimization
  on manifolds.
\newblock {\em SIAM Journal on Optimization}, 27(1):173--189, 2017.

\bibitem{Hu2024constraint}
X.~Hu, N.~Xiao, X.~Liu, and K.-C. Toh.
\newblock A constraint dissolving approach for nonsmooth optimization over the
  {S}tiefel manifold.
\newblock {\em IMA Journal of Numerical Analysis}, 44(6):3717--3748, 2024.

\bibitem{Huang2022riemannian}
W.~Huang and K.~Wei.
\newblock {R}iemannian proximal gradient methods.
\newblock {\em Mathematical Programming}, 194(1-2):371--413, 2022.

\bibitem{Kantorovich1958space}
L.~V. Kantorovich and S.~Rubinshtein.
\newblock On a space of totally additive functions.
\newblock {\em Vestnik of the St. Petersburg University: Mathematics},
  13(7):52--59, 1958.

\bibitem{Kuhn2019wasserstein}
D.~Kuhn, P.~M. Esfahani, V.~A. Nguyen, and S.~Shafieezadeh-Abadeh.
\newblock {W}asserstein distributionally robust optimization: Theory and
  applications in machine learning.
\newblock In {\em Operations Research $\&$ Management Science in the Age of
  Analytics}, pages 130--166. INFORMS, 2019.

\bibitem{Kuhn2024distributionally}
D.~Kuhn, S.~Shafiee, and W.~Wiesemann.
\newblock Distributionally robust optimization.
\newblock {\em arXiv:2411.02549}, 2024.

\bibitem{Lai2014splitting}
R.~Lai and S.~Osher.
\newblock A splitting method for orthogonality constrained problems.
\newblock {\em Journal of Scientific Computing}, 58(2):431--449, 2014.

\bibitem{Li2021weakly}
X.~Li, S.~Chen, Z.~Deng, Q.~Qu, Z.~Zhu, and M.-C.~A. So.
\newblock Weakly convex optimization over {S}tiefel manifold using {R}iemannian
  subgradient-type methods.
\newblock {\em SIAM Journal on Optimization}, 31(3):1605--1634, 2021.

\bibitem{Liu2024penalty}
X.~Liu, N.~Xiao, and Y.-X. Yuan.
\newblock A penalty-free infeasible approach for a class of nonsmooth
  optimization problems over the {S}tiefel manifold.
\newblock {\em Journal of Scientific Computing}, 99(2):30, 2024.

\bibitem{Lu2012augmented}
Z.~Lu and Y.~Zhang.
\newblock An augmented {L}agrangian approach for sparse principal component
  analysis.
\newblock {\em Mathematical Programming}, 135:149--193, 2012.

\bibitem{Rockafellar1970convex}
R.~T. Rockafellar.
\newblock {\em Convex Analysis}.
\newblock Princeton University Press, Princeton, 1970.

\bibitem{Rockafellar2009variational}
R.~T. Rockafellar and R.~J.-B. Wets.
\newblock {\em Variational Analysis}.
\newblock Springer Science \& Business Media, 2009.

\bibitem{Si2024riemannian}
W.~Si, P.-A. Absil, W.~Huang, R.~Jiang, and S.~Vary.
\newblock A {R}iemannian proximal {N}ewton method.
\newblock {\em SIAM Journal on Optimization}, 34(1):654--681, 2024.

\bibitem{Van2021data}
B.~P. Van~Parys, P.~M. Esfahani, and D.~Kuhn.
\newblock From data to decisions: Distributionally robust optimization is
  optimal.
\newblock {\em Management Science}, 67(6):3387--3402, 2021.

\bibitem{Wang2024decentralized}
L.~Wang, L.~Bao, and X.~Liu.
\newblock A decentralized proximal gradient tracking algorithm for composite
  optimization on {R}iemannian manifolds.
\newblock {\em arXiv:2401.11573}, 2024.

\bibitem{Wang2021multipliers}
L.~Wang, B.~Gao, and X.~Liu.
\newblock Multipliers correction methods for optimization problems over the
  {S}tiefel manifold.
\newblock {\em CSIAM Transactions on Applied Mathematics}, 2(3):508--531, 2021.

\bibitem{Wang2023communication}
L.~Wang, X.~Liu, and Y.~Zhang.
\newblock A communication-efficient and privacy-aware distributed algorithm for
  sparse {PCA}.
\newblock {\em Computational Optimization and Applications}, 85(3):1033--1072,
  2023.

\bibitem{Wang2024seeking}
L.~Wang, X.~Liu, and Y.~Zhang.
\newblock Seeking consensus on subspaces in federated principal component
  analysis.
\newblock {\em Journal of Optimization Theory and Applications}, 203:529--561,
  2024.

\bibitem{Wang2022manifold}
Z.~Wang, B.~Liu, S.~Chen, S.~Ma, L.~Xue, and H.~Zhao.
\newblock A manifold proximal linear method for sparse spectral clustering with
  application to single-cell {RNA} sequencing data analysis.
\newblock {\em INFORMS Journal on Optimization}, 4(2):200--214, 2022.

\bibitem{Wen2013feasible}
Z.~Wen and W.~Yin.
\newblock A feasible method for optimization with orthogonality constraints.
\newblock {\em Mathematical Programming}, 142(1):397--434, 2013.

\bibitem{Wiesemann2013robust}
W.~Wiesemann, D.~Kuhn, and B.~Rustem.
\newblock Robust {M}arkov decision processes.
\newblock {\em Mathematics of Operations Research}, 38(1):153--183, 2013.

\bibitem{Xiao2021exact}
N.~Xiao, X.~Liu, and Y.-X. Yuan.
\newblock Exact penalty function for $\ell_{2,1}$ norm minimization over the
  {S}tiefel manifold.
\newblock {\em SIAM Journal on Optimization}, 31(4):3097--3126, 2021.

\bibitem{Xu2018distributionally}
H.~Xu, Y.~Liu, and H.~Sun.
\newblock Distributionally robust optimization with matrix moment constraints:
  {L}agrange duality and cutting plane methods.
\newblock {\em Mathematical Programming}, 169:489--529, 2018.

\bibitem{Yang2014optimality}
W.~H. Yang, L.-H. Zhang, and R.~Song.
\newblock Optimality conditions for the nonlinear programming problems on
  {R}iemannian manifolds.
\newblock {\em Pacific Journal of Optimization}, 10(2):415--434, 2014.

\bibitem{Zhang2020primal}
J.~Zhang, S.~Ma, and S.~Zhang.
\newblock Primal-dual optimization algorithms over {R}iemannian manifolds: An
  iteration complexity analysis.
\newblock {\em Mathematical Programming}, 184(1):445--490, 2020.

\bibitem{Zhang2024short}
L.~Zhang, J.~Yang, and R.~Gao.
\newblock A short and general duality proof for {W}asserstein distributionally
  robust optimization.
\newblock {\em Operations Research}, pages 1--10, 2024.

\end{thebibliography}

\addcontentsline{toc}{section}{References}

\end{document}